\documentclass[12pt]{amsart}

\usepackage{amsmath,amssymb}
\usepackage{xifthen}
\usepackage{psfrag,graphicx}
\usepackage{stmaryrd} 
\usepackage[left=25mm,right=25mm,top=25mm,bottom=25mm]{geometry}
\usepackage[T1]{fontenc}
\usepackage{color}
\date{\today}

\newtheorem{theorem}{Theorem}

\newtheorem{lemma}[theorem]{Lemma}
\newtheorem{corollary}[theorem]{Corollary}
\newtheorem{algorithm}[theorem]{Algorithm}

\newtheorem{remark}[theorem]{Remark}

\def\dual#1#2{\langle#1\hspace*{.5mm},#2\rangle}

\def\<{\langle\hspace*{-.9mm}\langle}
\def\>{\rangle\hspace*{-.9mm}\rangle}

\def\abs#1{\vert #1 \vert}
\newcommand{\norm}[3][]{#1\|#2#1\|_{#3}}
\def\snorm#1#2{|#1|_{#2}}
\newcommand{\enorm}[2][]{#1|\hspace*{-.3mm}#1|\hspace*{-.3mm}#1|#2#1|\hspace*{-.3mm}#1|\hspace*{-.3mm}#1|}

\def\div{{\rm div}}
\def\osc{{\rm osc}}
\def\ncon{{\rho}}

\def\wilde#1{\widetilde #1}
\def\wat#1{\widehat #1}

\def\nablag{{\nabla_\Gamma}}
\def\scurl{{\bf curl}}
\def\curl{{curl}}

\def\mesh{\TT}
\def\nodes{\NN}
\def\edges{\EE}
\def\el{T}

\def\ed{e}

\def\dist{{\rm{dist}}}

\def\elpatch{{\omega_{\el}}}

\def\slo{\mathcal{V}} 

\def\hyp{\mathcal{W}}

\def\CR{V}
\def\FE{\CR^0}
\def\jump#1{\llbracket#1\rrbracket}
\def\avg#1{\{#1\}}

\def\div{\textrm{div}}

\def\NC{\ensuremath{\perp}}
\def\CO{\ensuremath{0}}

\newcommand{\EE}{\ensuremath{\mathcal{E}}}

\newcommand{\MM}{\ensuremath{\mathcal{M}}}
\newcommand{\NN}{\ensuremath{\mathcal{N}}}
\newcommand{\OO}{\ensuremath{\mathcal{O}}}
\newcommand{\PP}{\ensuremath{\mathcal{P}}}

\newcommand{\TT}{\ensuremath{\mathcal{T}}}

\newcommand{\bN}{\ensuremath{\mathbb{N}}}

\newcommand{\R}{\ensuremath{\mathbb{R}}}

\newcommand{\bH}{\ensuremath{\mathbf{H}}}

\newcommand{\bL}{\ensuremath{\mathbf{L}}}

\newcommand{\V}{\ensuremath{\mathbf{V}}}

\newcommand{\bfc}{\ensuremath{\mathbf{c}}}

\newcommand{\bfm}{\ensuremath{\mathbf{m}}}
\newcommand{\n}{\ensuremath{\mathbf{n}}}

\newcommand{\bft}{\ensuremath{\mathbf{t}}}

\newcommand{\vv}{\ensuremath{\mathbf{v}}}

\newcommand{\x}{\ensuremath{\mathbf{x}}}
\newcommand{\y}{\ensuremath{\mathbf{y}}}

\newcounter{constantsnumber}
\def\setc#1{
  \ifthenelse{\equal{#1}{I}}{C_{\rm I}}{ 
  \ifthenelse{\equal{#1}{stab}}{C_{\rm stab}}{ 
  \ifthenelse{\equal{#1}{ell}}{C_{\rm ell}}{ 
  \ifthenelse{\equal{#1}{sat}}{C_{\rm sat}}{ 
  \ifthenelse{\equal{#1}{eff}}{C_{\rm eff}}{ 
  \ifthenelse{\equal{#1}{rel}}{C_{\rm rel}}{ 
  \ifthenelse{\equal{#1}{reduction}}{C_{\rm red}}{ 
  \ifthenelse{\equal{#1}{norm}}{C_{\rm norm}}{ 
  \ifthenelse{\equal{#1}{poinc}}{C_{\rm edge}}{ 
   \refstepcounter{constantsnumber}
   \label{const#1}C_{\theconstantsnumber}}}}}}}}}}}

\def\definec#1{\refstepcounter{constantsnumber}\label{const#1}}%

\def\c#1{
  \ifthenelse{\equal{#1}{I}}{C_{\rm I}}{ 
  \ifthenelse{\equal{#1}{stab}}{C_{\rm stab}}{ 
  \ifthenelse{\equal{#1}{ell}}{C_{\rm ell}}{ 
  \ifthenelse{\equal{#1}{sat}}{C_{\rm sat}}{ 
  \ifthenelse{\equal{#1}{eff}}{C_{\rm eff}}{ 
  \ifthenelse{\equal{#1}{rel}}{C_{\rm rel}}{ 
  \ifthenelse{\equal{#1}{reduction}}{C_{\rm red}}{ 
  \ifthenelse{\equal{#1}{norm}}{C_{\rm norm}}{ 
  \ifthenelse{\equal{#1}{poinc}}{C_{\rm edge}}{ 
    C_{\ref{const#1}}}}}}}}}}}}

\title[Adaptive CR-BEM]{Adaptive Crouzeix-Raviart Boundary Element Method}
\author[N.~Heuer]{Norbert Heuer}
\author[M.~Karkulik]{Michael Karkulik}
\address{Facultad de Matem\'aticas,
  Pontificia Universidad Cat\'olica de Chile,
  Avenida Vicu\~na Mackenna 4860, Santiago, Chile}
  \email{\{nheuer,\,mkarkulik\}@mat.puc.cl}
  \urladdr{http://www.mat.puc.cl/\{$\sim$nheuer,\,$\sim$mkarkulik\}}

\subjclass[2010]{65N30, 65N38, 65N50, 65R20}
\keywords{boundary element method, adaptive algorithm, nonconforming method, a posteriori error estimation}
\thanks{Financial support by CONICYT through projects
        Anillo ACT1118 (ANANUM) and Fondecyt 1110324, 3140614 is gratefully acknowledged.}
\begin{document}
\maketitle
\begin{abstract}
For the non-conforming Crouzeix-Raviart boundary elements from
[Heuer, Sayas: Crouzeix-Raviart boundary elements, Numer. Math. 112, 2009],
we develop and analyze a posteriori error estimators based on the $h-h/2$ methodology.
We discuss the optimal rate of convergence for uniform mesh refinement,
and present a numerical experiment with singular data where our adaptive algorithm recovers the optimal rate
while uniform mesh refinement is sub-optimal. We also discuss the case of reduced regularity by standard
geometric singularities to conjecture that, in this situation, non-uniformly refined meshes are not
superior to quasi-uniform meshes for Crouzeix-Raviart boundary elements.
\end{abstract}
\section{Introduction}\label{section:introduction}
This is the first paper on a posteriori error estimation and adaptivity for
an \emph{element-wise} non-conforming boundary element method,
namely Crouzeix-Raviart boundary elements analyzed in \cite{hs09}.
Previously, in \cite{DominguezH_PEA}, we presented an error estimate
for a boundary element method with non-conforming domain decomposition.
There, critical for the analysis is that the nonconformity of the method stems from
approximations that are discontinuous only across the interface of sub-domains,
which are assumed to be fixed. In that case, the underlying energy norm of
order $1/2$ of discrete functions has to be localized only with respect to sub-domains.
In this paper, where we consider approximations which are discontinuous across
element edges, such sub-domain oriented arguments do not apply. Instead, we have
to find localization arguments that are uniform under scalings with $h$,
the diameter of elements, which is nontrivial in fractional order Sobolev spaces
of order $\pm 1/2$.

The Crouzeix-Raviart boundary element method is of particular theoretical interest since it serves
to set the mathematical foundation of (locally) non-conforming elements
for the approximation of hypersingular integral equations.
Our main theoretical result is the efficiency and reliability
(based on a saturation assumption) of several a posteriori error estimators.
Our second result is that, for problems with standard geometric singularities,
Crouzeix-Raviart boundary elements with seemingly appropriate mesh refinement
is as good as (and not better than) Crouzeix-Raviart boundary elements on
quasi-uniform meshes. We further discuss this point below.

The a posteriori error estimators in this work are based on the $h-h/2$-strategy.
This strategy is well known from ordinary differential equations~\cite{hnw} and finite element methods~\cite{ao00,bank}.
Recently, it was applied to conforming boundary element methods~\cite{fp08,effp09} as well:
If the discrete space $X_\ell$ is used to approximate the function $\phi$ in the energy norm $\enorm{\cdot}$, we use
the uniformly refined space $\wat X_\ell$ and the corresponding approximations $\Phi_\ell$ and $\wat\Phi_\ell$ to estimate the error via
the heuristics
\begin{align}\label{eq:hh2}
  \eta_\ell := \enorm{\wat\Phi_\ell-\Phi_\ell} \sim \enorm{\phi-\Phi_\ell}.
\end{align}
In a conforming setting, the proof of efficiency of $\eta_\ell$ (i.e, it bounds the error from below) follows readily from orthogonality properties,
while its reliability (i.e., it is an upper bound of the error) is additionally based on a saturation assumption.
In non-conforming methods, orthogonality is available only in a weaker form which contains additional terms,
such that $h-h/2$-based estimators are more involved than in a conforming setting.

As mentioned before, additional difficulties arise in boundary element methods due to the fact that
the underlying energy norm $\enorm{\cdot}$ is equivalent to a fractional order Sobolev norm.
These norms typically cannot be split into local error indicators. We use ideas from~\cite{fp08} to localize via weighted integer order Sobolev norms.

We are particularly interested in problems with singularities which are inherent
to problems on polyhedral surfaces where corner and corner-edge singularities
appear. In the extreme case of the hypersingular integral equation on a plane open
surface $\Gamma$ (which is our model problem), its solution is not in $H^1(\Gamma)$
since edge singularities behave like the square root of the distance to the boundary curve
\cite{Stephan_87_BIE}. The energy norm of this problem defines a Sobolev space of
order $1/2$, so that low-order conforming methods with quasi-uniform meshes
have approximation orders equal to $\OO(h^{1/2})$ ($h$ being the mesh size),
cf. \cite{BespalovH_08_hpB}. In \cite{hs09} the authors have
shown that this is also true for Crouzeix-Raviart boundary elements. Now, for an
adaptive method or a method with appropriate mesh refinement towards the singularities,
one expects to recover the optimal rate $\OO(h)$ of a low-order method.
{\em Surprisingly, this appears to be false in the case of
Crouzeix-Raviart boundary elements.}

We conjecture that $\OO(h^{1/2})$ (or $\OO(N^{-1/4})$ with $N$ being the number of unknowns)
is the optimal rate for our model problem even when using non-uniformly refined meshes.
We base our conjecture on two observations.
Standard error estimation of non-conforming methods, based on the second Strang lemma,
comprise a best-approximation term and a nonconformity term.
The best-approximation term has indeed the optimal order of a conforming method
but we observe that
the standard upper bound of the nonconformity term is of the order $\OO(N^{-1/4})$ and not better.
This surprising result can be explained by the fact that
the appearing Lagrangian multipliers on the edges of the elements
(needed for the jump condition of the Crouzeix-Raviart basis functions) are approximated
in a Sobolev space of order only $1/2$ less than the unknown function. Taking
into account that the total relative measure of the edges increases with mesh refinement
and that the Lagrangian multipliers are approximated only by constants, this
explains the limited convergence order of the whole method.

The second observation stems from numerical experiments with Crouzeix-Raviart boundary elements
using meshes which are optimal for \emph{conforming} methods:
\begin{itemize}
  \item We consider uniform meshes for the non-conforming approximation of
        a solution which is an element of the coarsest conforming space
        (i.e., a conforming method would compute the exact solution).
  \item We consider algebraically graded meshes which are optimal for conforming
        approximations in the sense that they guarantee an approximation order
        $\OO(N^{-1/2})$ for inherent singularities.
\end{itemize}
Both types of meshes show the reduced order of convergence $\OO(N^{-1/4})$, and the same reduced order is observed for our adaptive procedure.

Based on this conjecture, we conclude that, for Crouzeix-Raviart boundary elements,
quasi-uniform meshes are optimal to approximate standard geometric
singularities where the solution is almost in $H^1(\Gamma)$. There is no need for
adaptive mesh refinement. In this case, the only use of a posteriori error estimation
is the very error estimation.

There are cases, however, where given data are singular so that solutions have
singular behavior which is stronger than that due to geometric irregularities.
In these cases an adaptive Crouzeix-Raviart boundary element method can be used
to recover the optimal rate $\OO(N^{-1/4})$ which cannot be achieved with quasi-uniform meshes
in this situation. Our numerical experiments report on such a case where the exact solution
is strictly less regular than $H^1(\Gamma)$.\\

As model problem, we use the Laplacian exterior to a polyhedral domain or an open polyhedral surface. The Neumann problem for such a problem
can be written equivalently with the hypersingular integral operator $\hyp$,
\begin{align}\label{eq:hyp}
  \hyp \phi (\x) := - \frac{1}{4 \pi} \partial_{\n(\x)}\int_\Gamma \partial_{\n(\y)}\left( \frac{1}{\abs{\x-\y}} \right) \phi(\y)\,d\Gamma(\y) = f(\x),
\end{align}
where $\Gamma$ is the open or closed surface and $f$ is a given function. The link to the Neumann problem for the exterior Laplacian is given by
the special choice $f = (1/2-K')v$, with $v$ the Neumann datum and $K'$ the adjoint
of the double-layer operator.
Although the operator $\hyp$ can act on discontinuous functions, the hypersingular integral equation~\eqref{eq:hyp} is not well-posed in such a case.
However, continuity requirements can be relaxed by using the relation
$\hyp = \curl_\Gamma \slo \scurl_\Gamma$ with single layer operator $\slo$ and
certain surface differential operators $\curl_\Gamma$ and $\scurl_\Gamma$, see~\cite{Nedelec_82_IEN,ghh:09}.
This identity allows us to use the space $V$ of Crouzeix-Raviart elements to approximate
the exact solution $\phi$ of~\eqref{eq:hyp} in a non-conforming way.
The associated energy norm will then be $\enorm{\cdot} = \norm{\scurl \cdot}{H^{-1/2}(\Gamma)}$, see Section~\ref{section:crbem}.

The reliability and efficiency of $h-h/2$ error estimators for conforming methods follows readily from the Galerkin orthogonality
\begin{align}\label{eq:orth}
  \enorm{\phi-\Phi_\ell}^2 = \enorm{\phi-\Phi_\ell}^2 + \enorm{\wat\Phi_\ell - \Phi_\ell}^2,
\end{align}
where reliability additionally needs the saturation assumption
\begin{align*}
  \enorm{\phi-\wat\Phi_\ell} \leq \c{sat} \enorm{\phi-\Phi_\ell}, \quad \text{ with } 0<\c{sat}<1 \text{ for all } \ell\in\mathbb{N}.
\end{align*}
In a non-conforming setting, the orthogonality~\eqref{eq:orth} does not hold true any longer. However, there is a substitute given by an estimate
which involves additional terms of the form
$\enorm{\Phi_\ell - \Phi_\ell^\CO}$, with $\Phi_\ell^\CO$ being a conforming approximation of $\phi$, see Section~\ref{section:conform}.

A term of the form $\enorm{\Phi_\ell - \Phi_\ell^\CO}$ will be called \textit{nonconformity error}. Although it is computable, it is evident that the
computation of $\Phi_\ell^\CO$ has to be avoided. Hence, we will show that the nonconformity error can be bounded by inter-element jumps of $\Phi_\ell$, see Corollary~\ref{cor:ncon}.
To that end, we will analyze the properties of quasi-interpolation operators in the space $H^{-1/2}(\Gamma)$ in Section~\ref{section:interpolation}.

In Section~\ref{section:aposteriori}, we show that the a posteriori error estimator $\eta_\ell$ from~\eqref{eq:hh2} is reliable and efficient up to the nonconformity error, which
can then be exchanged with the inter-element jumps of $\Phi_\ell$. As already mentioned, $\eta_\ell$ is not localized, and we will use ideas from~\cite{fp08} to introduce three
additional error estimators for that purpose. Two of them are localized, see Section~\ref{section:local:est}, and can be used in a standard adaptive algorithm,
see Algorithm~\ref{algorithm} below.
We show in Section~\ref{section:aposteriori} that all error estimators are efficient and, under the saturation assumption, also reliable, up to inter-element jumps.
Finally, Section~\ref{section:numerics} presents numerical results.

\section{Crouzeix-Raviart boundary elements}\label{section:model}
\subsection{Notation and model problem}\label{section:notation}
We consider an open, plane, polygonal screen $\Gamma\subset\R^2$, embedded in $\R^3$, with normal $\n(\y)$ at $\y\in\Gamma$ pointing upwards. Restricting ourselves to a plane screen simplifies
the presentation. However, associated solutions exhibit the strongest possible edge singularities that,
at least for conforming methods, require nonuniform meshes in order to guarantee efficiency of approximation. On $\Gamma$, we use the standard
spaces $L_2(\Gamma)$ and $H^1(\Gamma)$, and as usual, $H^1_0(\Gamma)\subset H^1(\Gamma)$ consists of functions
that vanish on the boundary $\partial\Gamma$. The space $H^1_0(\Gamma)$ is equipped with the $H^1(\Gamma)$
(semi-)norm $\abs{\cdot}_{H^1(\Gamma)} := \norm{\nablag\cdot}{L_2(\Gamma)}$ where $\nablag$ denotes
the surface gradient.
We define intermediate spaces by the $K$-method of interpolation (see, e.g.,~\cite{triebel}), that is,
\begin{align*}
  H^s(\Gamma) = \left[ L_2(\Gamma), H^1(\Gamma) \right]_{s} \quad\text{ and } \quad\wilde H^s(\Gamma) = \left[ L_2(\Gamma), H^1_0(\Gamma) \right]_{s} \quad\text{ for } 0 < s < 1.
\end{align*}
Sobolev spaces with negative index are defined via duality with respect to the extended $L_2(\Gamma)$ inner product $\dual{\cdot}{\cdot}$,
\begin{align*}
  H^s(\Gamma) := \wilde H^{-s}(\Gamma)' \quad\text{ and } \quad \wilde H^s(\Gamma) := H^{-s}(\Gamma)' \quad\text{ for } -1\leq s < 0.
\end{align*}
Space of vector valued functions will be denoted by bold-face letters, i.e. $\bL_2(\Gamma)$ or $\bH^{1/2}(\Gamma)$, meaning
that every component is an element of the respective space.
We will use tangential differential operators.
For sufficiently smooth functions $\phi$ on $\Gamma$, we define the tangential curl operator $\scurl$ by
\begin{align*}
  \scurl \phi := \left( \partial_y \phi, -\partial_x \phi, 0 \right).
\end{align*}
Drawing upon the results from~\cite{bcs02}, it is shown in~\cite[Lemma 2.2]{ghh:09} that
the operator $\scurl$ can be extended to a continuous operator,
mapping $\wilde H^{1/2}(\Gamma)$ to
\begin{align*}
  \wilde\bH^{-1/2}(\Gamma) := \big\{ \psi \in \big( \wilde H^{-1/2}(\Gamma) \big)^3 \mid \psi\cdot\n = 0 \big\}.
\end{align*}
Now our model problem is as follows.
For a given $f\in H^{-1/2}(\Gamma)$, {\em find $\phi\in \wilde H^{1/2}(\Gamma)$ such that}
\begin{align}\label{eq:hypsing}
  \dual{\hyp \phi}{\psi} = \dual{f}{\psi} \quad \text{ for all } \psi\in\wilde H^{1/2}(\Gamma).
\end{align}
Here, $\hyp$ is the hypersingular integral operator from~\eqref{eq:hyp}. It is well known that this problem
has a unique solution, cf. \cite{Stephan_87_BIE}. Recall the relation
$\hyp = \curl_\Gamma \slo \scurl_\Gamma$ with single layer operator $\slo$,
\begin{align*}
  \slo u (\x) := \frac{1}{4 \pi} \int_\Gamma \frac{1}{\abs{\x-\y}} u(\y)\,d\Gamma(\y).
\end{align*}
Performing integration by parts one finds that an equivalent formulation of~\eqref{eq:hypsing} is given by
\begin{align}\label{eq:symm}
  \dual{\slo \scurl \phi}{\scurl\psi} = \dual{f}{\psi} \quad \text{ for all } \psi\in\wilde H^{1/2}(\Gamma),
\end{align}
see~\cite{Nedelec_82_IEN} and \cite[Lemma 2.3]{ghh:09}.
Note that $\slo$ in~\eqref{eq:symm} is considered to transfer \textit{vectorial} densities into \textit{vectorial} potentials, i.e., $\slo$ acts component-wise.
\subsection{Meshes and local mesh-refinement}\label{section:meshes}
A triangulation $\mesh$ of $\Gamma$ consists of compact 2-dimensional simplices (i.e., triangles) $\el$ such that
$\bigcup_{\el\in\mesh}\el = \overline\Gamma$. We do not allow hanging nodes. The volume area $\abs{\el}$ of every
element defines the local mesh-width $h_\mesh\in L_\infty(\Gamma)$ by $h_\mesh|_\el := h_\mesh(\el) := \abs{\el}^{1/2}$.
We define $\edges_\mesh$ to be the set of all edges $\ed$ of the triangulation $\mesh$, and $\nodes_\mesh$ as
the set of all nodes $z$ of the triangulation which are not on the boundary $\partial\Gamma$.
We will need different kinds of patches. For a node $z\in\nodes_\mesh$, we denote by $\omega_z$ the node patch as the
set of all elements $\el\in\mesh$ sharing $z$. Likewise, we define an edge patch $\omega_\ed$. For an element $\el\in\mesh$,
the patch $\omega_\el$ is the set of all elements sharing a node with $\el$.

Starting from an initial triangulation $\mesh_0$ of $\Gamma$, we will generate a sequence
of meshes $\mesh_\ell$ for $\ell\in\bN$ via so-called \textit{newest vertex bisection} (NVB).
For a brief overview, we refer to Figure~\ref{fig:nvb}, and for a precise definition,
we refer the reader to~\cite{verfuerth,kpp}.
We denote by $\overline\el$ a fixed reference element, and by $\overline u$ the pull-back of a function $u$ defined on $\el$,
i.e., if $F_\el:\overline\el\rightarrow\el$ is the affine element map, $\overline u := u \circ F_\el$.
An important property of the NVB refinement strategy is that one can not only map elements $\el$ to fixed reference domains,
but also patches.
This means that there is a finite set of fixed reference patches and affine maps such that any node-, element-, or
edge patch is the affine image of such a reference patch.
In particular, there are only finitely many constants involved in scaling argument on patches, and hence,
one may use patches in scaling arguments.
For a mesh $\mesh$, we denote by $\wat\mesh$ the uniformly refined mesh, i.e., all edges in $\mesh$ are bisected.

\begin{figure}[t]
\centering
\psfrag{T0}{}
\psfrag{T1}{}
\psfrag{T2}{}
\psfrag{T3}{}
\psfrag{T4}{}
\psfrag{T12}{}
\psfrag{T34}{}
\includegraphics[width=35mm]{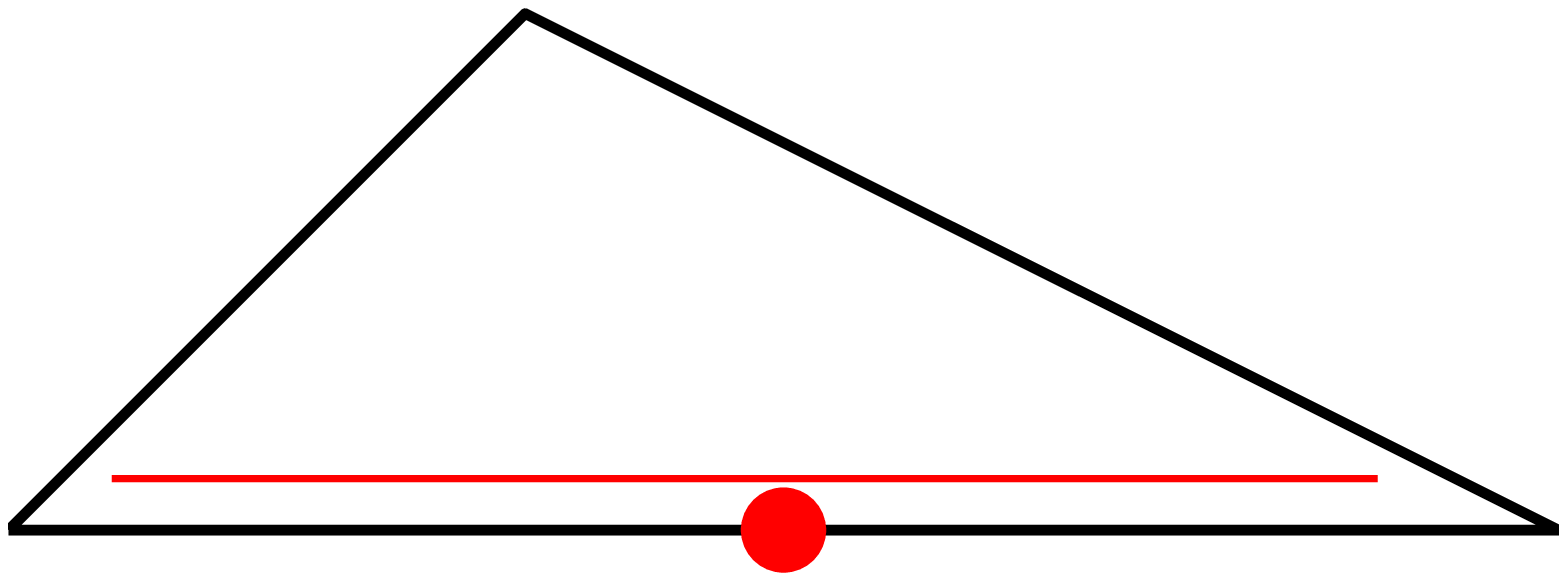} \quad
\includegraphics[width=35mm]{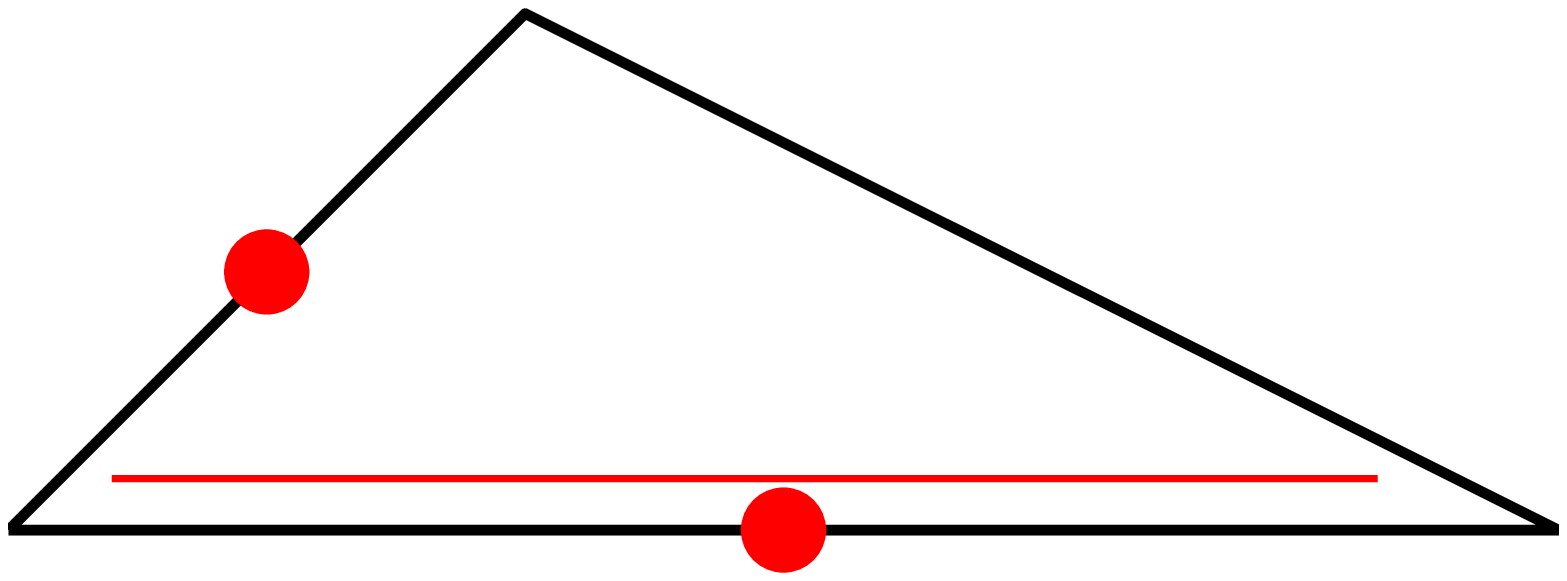} \quad
\includegraphics[width=35mm]{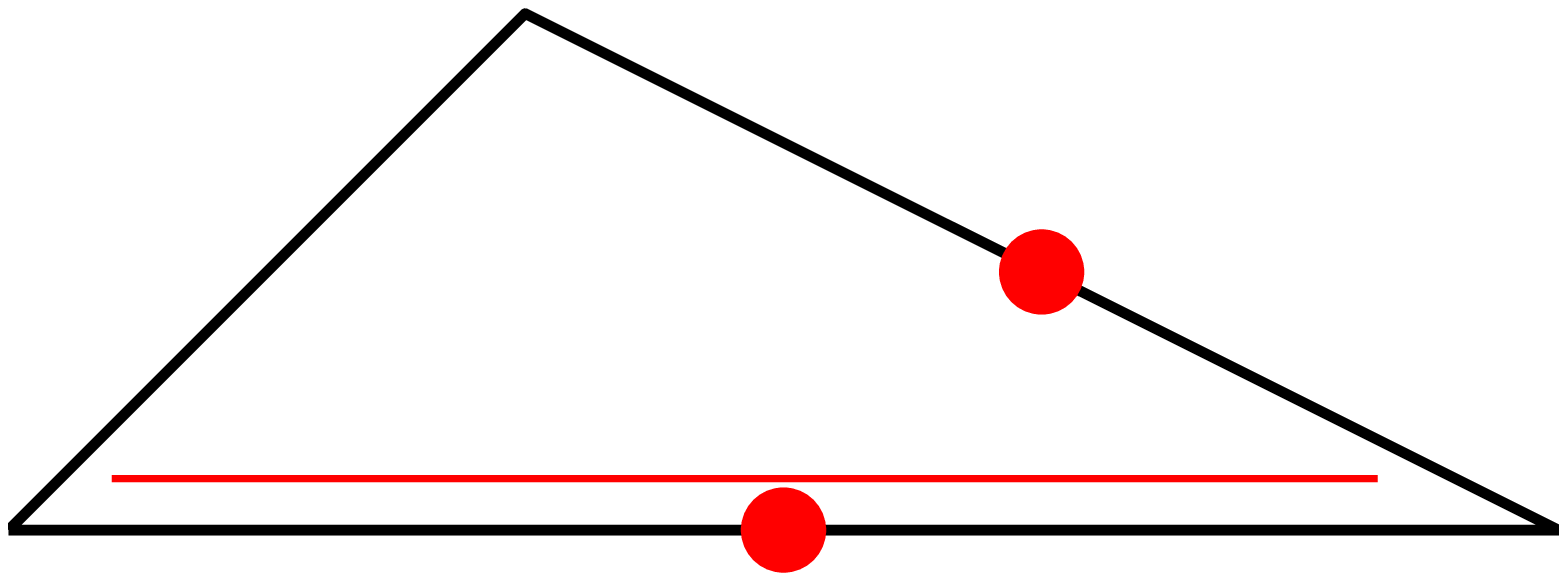} \quad
\includegraphics[width=35mm]{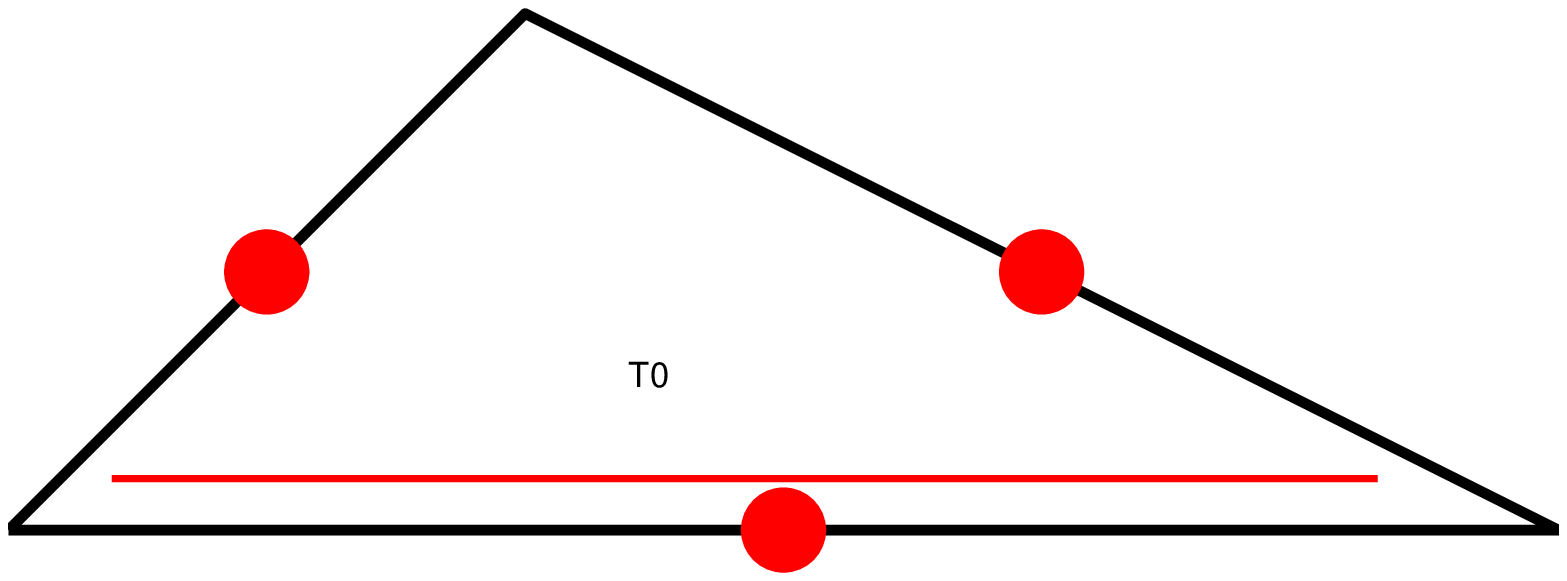} \\
\includegraphics[width=35mm]{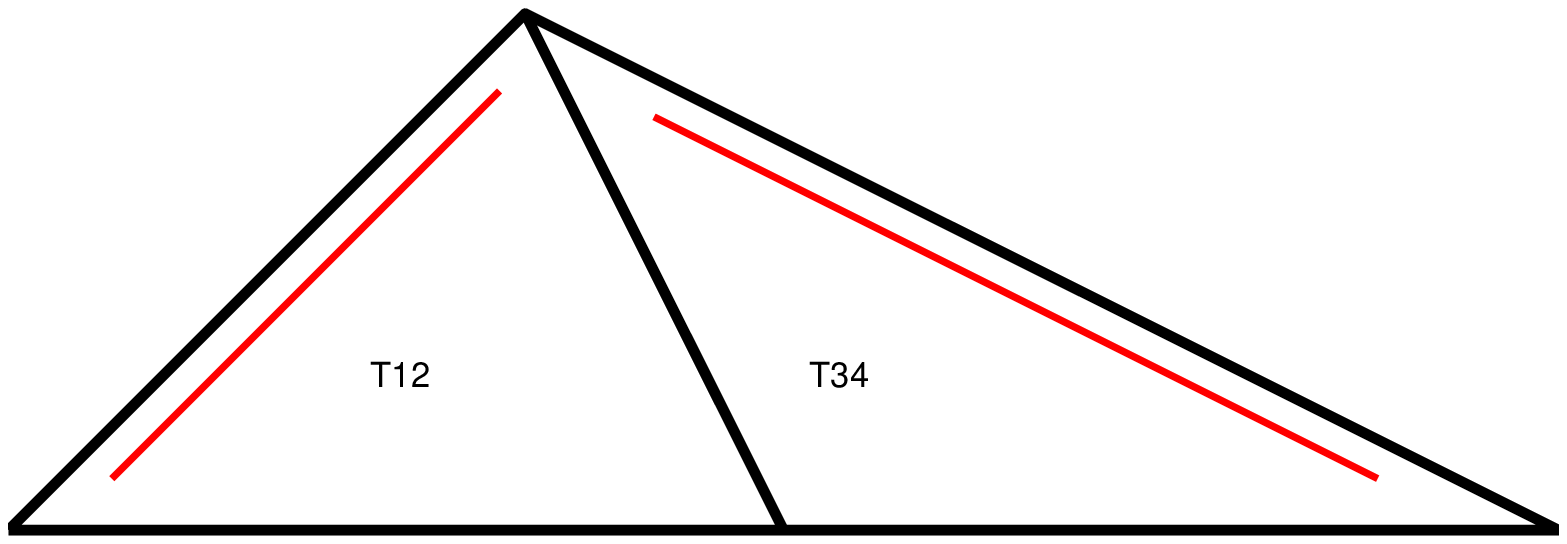} \quad
\includegraphics[width=35mm]{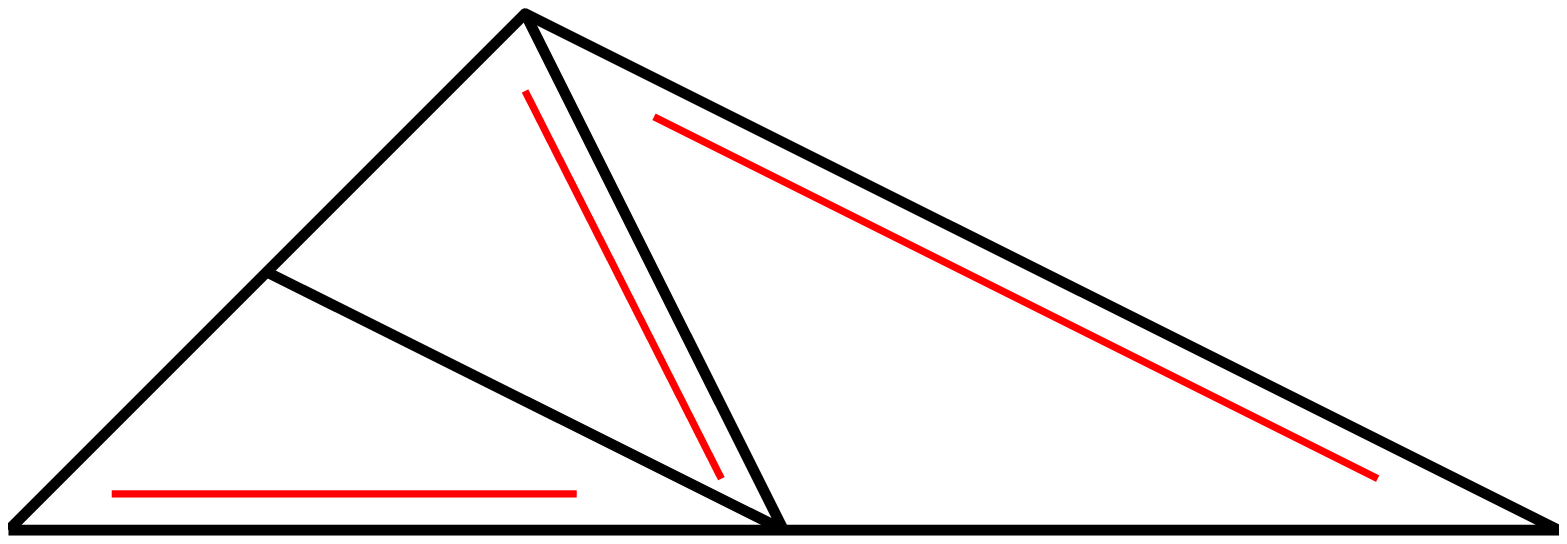}\quad
\includegraphics[width=35mm]{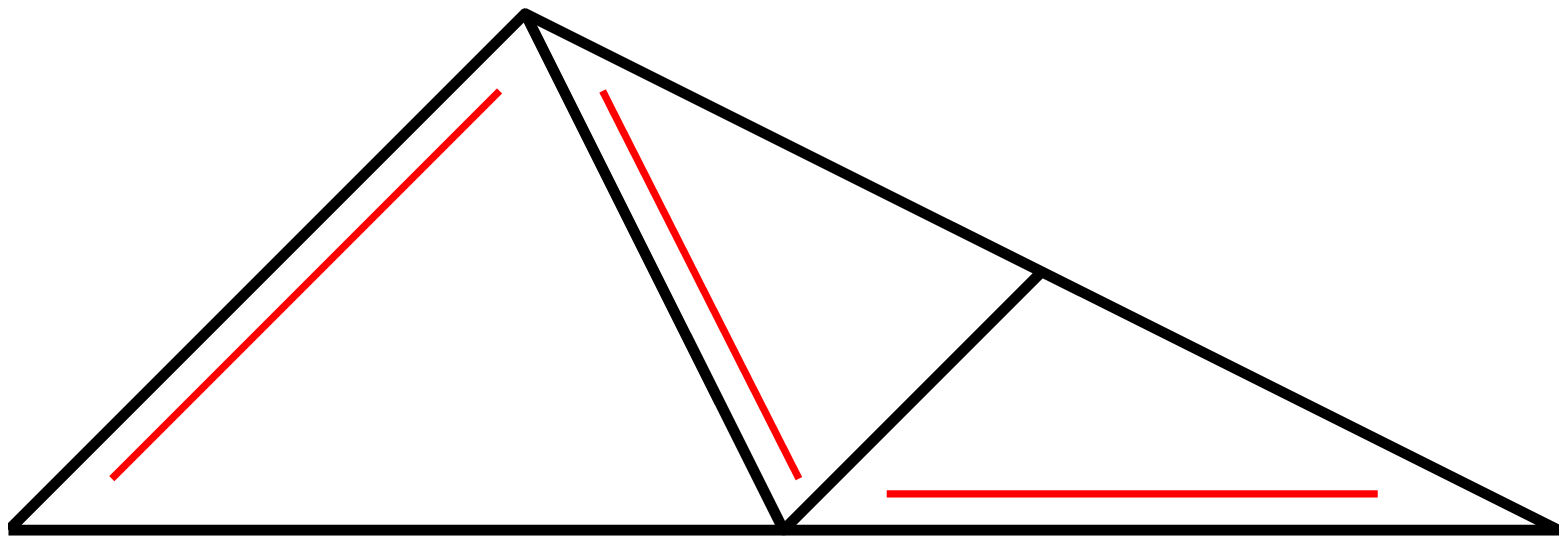}\quad
\includegraphics[width=35mm]{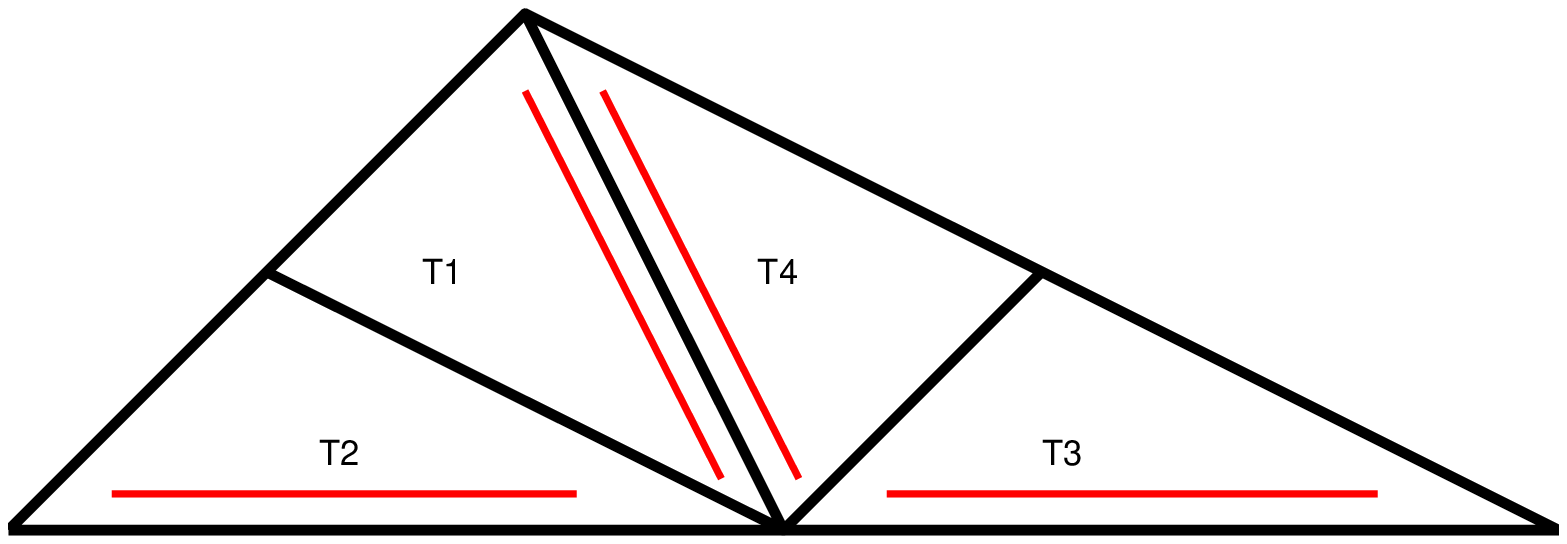}
\caption{
For each triangle $\el\in\mesh_\ell$, there is one fixed \emph{reference edge},
indicated by the double line (left, top). Refinement of $\el$ is done by bisecting
the reference edge, where its midpoint becomes a new node. The reference
edges of the son triangles $\el'\in\mesh_{\ell+1}$ are opposite to this newest vertex (left, bottom).
To avoid hanging nodes, one proceeds as follows:
We assume that certain edges of $\el$, but at least the reference edge,
are marked for refinement (top).
Using iterated newest vertex bisection, the element is then split into
2, 3, or 4 son triangles (bottom).
If all elements are refined by three bisections (right, bottom), we obtain the
so-called uniform bisec(3)-refinement which is denoted by $\widehat\mesh_\ell$.}
\label{fig:nvb}
\end{figure}

For a triangle $\el\in\mesh$, we denote by $\n_\el$ the normal vector on $\partial\el$ pointing outwards of $\el$.
For an inner edge $\ed\in\edges_\mesh$, i.e., $\ed\subset\Gamma$, we denote by $\el_\ed^+$ and $\el_\ed^-$ the two elements of $\mesh$ sharing $\ed$,
and we define $\n^+ := \n_{\el_\ed^+}$ and $\n^-:=\n_{\el_\ed^-}.$
For smooth enough
functions $\phi:\Gamma\rightarrow\R$ and $\vv:\Gamma\rightarrow\R^2$ we define the jumps $\jump{\cdot}$ and averages $\avg{\cdot}$ of the traces $\phi^+$, $\phi^-$, $\vv^+$, and $\vv^-$
by
\begin{align*}
  \begin{array}{ll}
    \avg{\phi}|_\ed := \frac{1}{2} (\phi^+ + \phi^-), & \avg{\vv}|_\ed := \frac{1}{2}(\vv^+ + \vv^-), \\
    \jump{\phi}|_\ed := \phi^+\n^+ + \phi^-\n^-, & \jump{\vv}|_\ed := \vv^+\n^+ + \vv^-\n^-.
  \end{array}
\end{align*}
If we equip a mesh with an index, e.g., $\mesh_\ell$, then we will use the index $(\cdot)_\ell$ instead of $(\cdot)_{\mesh_\ell}$, i.e., we write, e.g., $h_\ell$ instead of $h_{\mesh_\ell}$,
and the same abbreviation will be used for sets of edges or nodes, e.g., $\edges_\ell$ or $\nodes_\ell$.
\subsection{Crouzeix-Raviart boundary elements}\label{section:crbem}
For a given mesh $\mesh$, $\PP^1(\mesh)$ is the space of piecewise linear functions.
By $\FE = \FE_\mesh$, we denote the space of lowest-order continuous boundary elements, i.e.,
\begin{align*}
  \FE := \PP^1(\mesh)\cap H^1_0(\Gamma),
\end{align*}
and $\CR = \CR_\mesh$ is the space of Crouzeix-Raviart boundary elements, i.e.,
\begin{align*}
  \CR := \left\{ \Phi \in \PP^1(\mesh)\:\: \vline \:\:
    \begin{aligned}
      \Phi \text{ is continuous in } \bfm_\ed\: &\forall\ed\in\edges_\mesh \text{ with }\ed\nsubseteq\partial\Gamma,\\
      \Phi(\bfm_\ed) = 0\: &\forall\ed\in\edges_\mesh \text{ with } \ed\subset\Gamma
    \end{aligned}
  \right\},
\end{align*}
where $\bfm_\ed$ is the midpoint of $\ed\in\edges_\mesh$.
For $\scurl_\mesh:\PP^1(\mesh)\rightarrow \bL_2(\Gamma)$ being the $\mesh$-piecewise tangential curl operator,
a norm in $\CR$ is given by
\begin{align*}
  \enorm{\cdot}_{\mesh} := \norm{\scurl_\mesh\cdot}{\wilde\bH^{-1/2}(\Gamma)}.
\end{align*}
In the following we consider the bilinear form
\begin{align*}
  a_\mesh(\Phi,\Psi):=\dual{\slo \scurl_\mesh \Phi}{\scurl_\mesh \Psi}.
\end{align*}
By the properties of the single-layer operator $\slo$, cf.~\cite{mclean}, $a_\mesh$ is symmetric and there
is a constant $\setc{norm}>1$, independent of $\mesh$ and $\Phi\in\CR$, such that
\begin{align*}
  \c{norm}^{-2} \enorm{\Phi}_{\mesh}^2 \leq a_\mesh(\Phi,\Phi) \leq \c{norm}^2 \enorm{\Phi}_{\mesh}^2.
\end{align*}
This makes $a_\mesh$ an inner product in $V$, which is therefore a Hilbert space.
Assuming additional regularity $f\in H^{-1/2+\varepsilon}(\Gamma)$ with $\varepsilon>0$, then
\begin{align}\label{eq:rhs}
  \dual{f}{\Psi} \leq \norm{f}{H^{-1/2+\varepsilon}(\Gamma)} \norm{\Psi}{H^{1/2-\varepsilon}(\Gamma)}
                 \leq C_\mesh \enorm{\Psi}_\mesh \quad \text{ for all } \Psi\in\CR.
\end{align}
Here we used the equivalence of norms in the finite-dimensional space $\CR$,
such that the number $C_\mesh>0$ depends on $\mesh$.
\definec{stab}\definec{ell}
By the Lax-Milgram lemma there exists a unique Galerkin solution $\Phi\in\CR$ of
\begin{align}\label{eq:galerkin}
  \dual{\slo\scurl_\mesh\Phi}{\scurl_\mesh\Psi} = \dual{f}{\Psi} \quad \text{ for all }\Psi\in\CR.
\end{align}
The unique solvability of~\eqref{eq:galerkin} was already addressed in~\cite{hs09} and studied via
an equivalent saddle-point problem.
We emphasize that the constant $C_\mesh$ in~\eqref{eq:rhs} depends on $\CR$, but is not
used in our analysis.
In the statements and arguments below, our notations will mostly omit the explicit dependence on
$\mesh$ by writing, e.g., $\enorm{\cdot}$, assuming that this is the norm related to the
finest mesh which occurs in the norms' argument.
\subsection{Uniform refinement: consistency error and optimal convergence}\label{section:uniform}
We briefly discuss existing results for the Crouzeix-Raviart BEM of Section~\ref{section:crbem} based on a sequence of uniformly refined meshes $(\mesh_\ell)_{\ell\in\mathbb{N}_0}$.
According to~\cite[Theorem 2]{hs09}, it holds that
\begin{align}\label{eq:uniform}
  \enorm{\phi-\Phi_\ell} \lesssim h_\ell^{1/2}\norm{\phi}{H^1(\Gamma)},
\end{align}
if $f\in L_2(\Gamma)$ and $(\mesh_\ell)_{\ell\in\mathbb{N}_0}$ is a uniform sequence of meshes with mesh width $h_\ell$. The proof of~\eqref{eq:uniform}
uses, as is customary in the analysis of non-conforming methods, the Lemma of Berger, Scott, and Strang, see, e.g.,~\cite{berger:scott:strang}.
With a view to the well-known approximation results of conforming method, it suffices to bound the so-called consistency error.
In~\cite[Prop. 5]{hs09}, it is shown that this can be done by
\begin{align*}
  \sup_{\Psi_\ell \in V_{\ell}} \frac{a(\phi-\Phi_\ell,\Psi_\ell)}{\norm{\scurl_\Gamma \Psi_\ell}{\wilde\bH^{-1/2}(\Gamma)}}
  \lesssim \inf_{\mu_\ell\in \PP^0(\edges_\ell)}\left[ \sum_{\ed\in\edges_\ell} \norm{\bft_\ed\cdot\slo\scurl\phi - \mu_\ell}{L_2(\ed)}^2 \right]^{1/2}
\end{align*}
and that the right-hand side converges like $\OO(h_\ell^{1/2})$, see~\cite[Lemma 6]{hs09}. However, this bound for the convergence rate of the right-hand side
is optimal. Indeed, for $v\in\PP^1(\Gamma)\setminus\PP^0(\Gamma)$ it holds that
\begin{align*}
  \inf_{\mu_\ell\in \PP^0(\edges_\ell)}\left[ \sum_{\ed\in\edges_\ell} \norm{v - \mu_\ell}{L_2(\ed)}^2 \right]^{1/2} \simeq \OO(h_\ell^{1/2}),
\end{align*}
which can be seen by a direct calculation.
Therefore, we are led to conjecture that the optimal order of convergence is $\OO(h^{1/2})$. An easy numerical example supports this conjecture. We choose
$\Gamma = \left[ 0,1 \right]^2$ and divide it along the diagonals and the midpoints of its sides, such that we obtain a mesh $\mesh_0$ of $8$ triangles.
We choose the exact solution $\phi\in\FE_{0}$ that vanishes on $\partial\Gamma$ and has the value $1$ in the center of $\Gamma$. In Figure~\ref{fig:unif_exact}, we
visualize the outcome of the corresponding Crouzeix-Raviart BEM based on a uniform mesh refinement.
We have not yet defined the shown quantities,
but what is important here is that $\Phi_\ell\in\CR_\ell$ denotes the Crouzeix-Raviart solution on the mesh $\mesh_\ell$, whereas $\Phi^0_\ell\in\FE_\ell$ denotes the conforming
solution. According to the definition of $\phi$, we have $\Phi^0_\ell=\phi$, and hence, according to~\eqref{eq:uniform},
\begin{align*}
  \enorm{\Phi_\ell-\Phi^0_\ell}=\enorm{\phi-\Phi_\ell} = \OO(h_\ell^{1/2}).
\end{align*}
One would expect an increased order of $\OO(h_\ell^{1-\varepsilon})$ for every $\varepsilon>0$, as $\phi\in\wilde H^{1/2}(\Gamma)\cap H^{3/2-\varepsilon}(\Gamma)$.
However, as Figure~\ref{fig:unif_exact} reveals, this increased rate is not achieved - we still observe $\OO(h_\ell^{1/2})$, which therefore seems to be the optimal rate
that can be obtained.

\begin{figure}[t]
\centering
\psfrag{muTilde}[cr][cr]{$\widetilde\mu_\ell^2$}
\psfrag{eta}[cr][cr]{$\eta_\ell^2$}
\psfrag{jumps}[cr][cr]{$\rho_\ell^2 + \widehat\rho_\ell^2$}
\psfrag{conforming error}[cr][cr]{$\enorm{\Phi_\ell - \Phi^0_\ell}^2$}
\psfrag{N^{-1/2}}[cr][cr]{$N^{-1/2}$}
\psfrag{N}[cc][cc]{Degrees of freedom}
\includegraphics{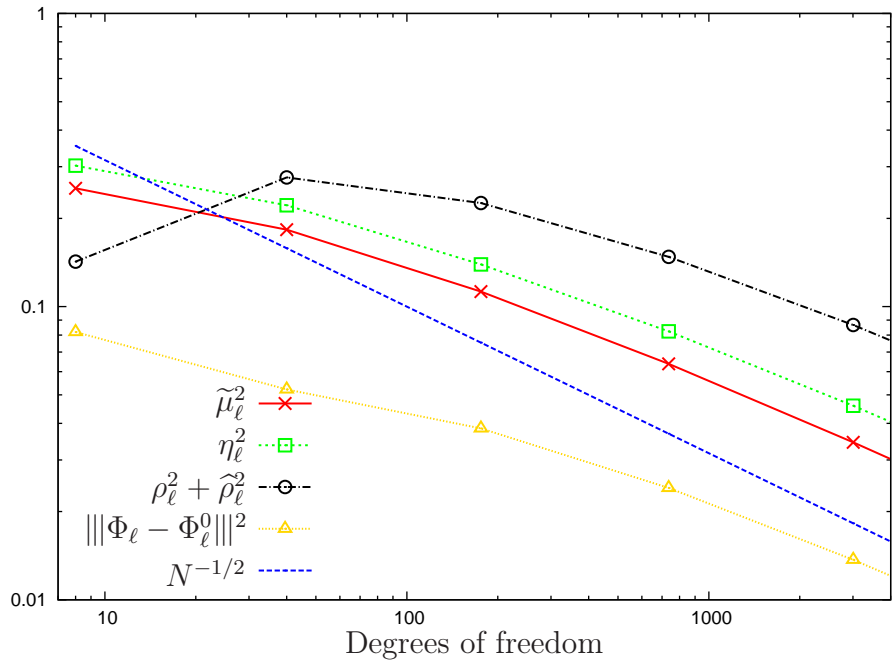}
\caption{Convergence rates for uniform mesh refinement and smooth solution. Note that we plot squared quantities, so that $\OO(N_\ell^{-1/2})$ corresponds to rate
of $\OO(h_\ell^{1/2})$ for the original quantities.}
\label{fig:unif_exact}
\end{figure}
\section{Preliminaries}\label{section:preliminaries}
\nopagebreak
\subsection{Conforming approximations and partial orthogonality}\label{section:conform}
For the development and analysis of the adaptive Crouzeix-Raviart boundary elements, it will be convenient to use
a decomposition of the space $\CR_\mesh$ into conforming and non-conforming components. Such a decomposition is
given by the identity
\begin{align*}
  \CR_\mesh = \FE_\mesh \oplus \CR^\perp_\mesh,
\end{align*}
where $\CR^\perp_\mesh$ is the orthogonal complement of $\FE_\mesh$ with respect to the inner product $a_\mesh(\cdot,\cdot)$.
For a function $\Phi\in\CR_\mesh$, we write $\Phi = \Phi^\CO + \Phi^\NC$ with $\Phi^\CO\in\FE_\mesh$ and $\Phi^\NC\in\CR^\perp_\mesh$.
We emphasize that there is a partial orthogonality, i.e., if $\mesh_\star$ is a refinement of $\mesh$, then
\begin{align*}
  a_\star(\phi-\Phi_\star,\Psi) = 0\quad\text{ for all } \Psi\in\FE_\mesh,
\end{align*}
where $\phi$ is the exact solution and $\Phi_\star\in\CR_{\mesh_\star}$ is its non-conforming Galerkin approximation.
In contrast to conforming methods, this orthogonality property cannot be extended to all $\Psi\in\CR_\mesh$.
However, it can be extended to a partial orthogonality as follows, cf.~\cite[Corollary 4.3]{bonito:nochetto}.
\begin{lemma}\label{lem:orth}
  Let $\mesh_\star$ be a refinement of $\mesh$ and $\Phi_\star\in\CR_{\mesh_\star}$ the Galerkin solution~\eqref{eq:galerkin} on $\mesh_\star$. 
  Then, for all $\varepsilon>0$,\definec{orth} and all $\Phi\in\CR_\mesh$, we have
  \begin{align*}
    a_\star(\phi-\Phi_\star,\phi-\Phi_\star) &\leq (1+\varepsilon)a(\phi-\Phi,\phi-\Phi)\\
    &\quad- \frac{\c{norm}^{-2}}{2}\enorm{\Phi-\Phi_\star}^2_\star
    + \Bigl(\c{norm}^2(1+\frac{1}{\varepsilon})+\c{norm}^{-2}\Bigr)
    \enorm{(\Phi-\Phi^\CO) - (\Phi_\star - \Phi_\star^\CO)}^2_\star
  \end{align*}
\end{lemma}
\begin{proof}
  As $\phi-\Phi_\star$ is orthogonal to $\FE_\mesh$ and $\FE(\mesh_\star)$, we have
  \begin{align*}
    a_\star(\phi-\Phi_\star,\phi-\Phi_\star)
    &= a_\star(\phi-\Phi_\star - \Phi^\CO+\Phi_\star^\CO,\phi-\Phi_\star-\Phi^\CO+\Phi_\star^\CO)
    - a_\star(\Phi_\star^\CO - \Phi^\CO,\Phi_\star^\CO-\Phi^\CO)\\
    &= a_\star(\phi-\Phi+\Phi^\NC-\Phi_\star^\NC,\phi-\Phi+\Phi^\NC-\Phi_\star^\NC)
    - a_\star(\Phi_\star^\CO - \Phi^\CO,\Phi_\star^\CO-\Phi^\CO)\\
    &= a_\star(\phi-\Phi,\phi-\Phi) + 2 a_\star(\Phi^\NC - \Phi_\star^\NC,\phi-\Phi)\\
    & \quad + a_\star(\Phi^\NC - \Phi_\star^\NC,\Phi^\NC-\Phi_\star^\NC) - a_\star(\Phi_\star^\CO - \Phi^\CO,\Phi_\star^\CO-\Phi^\CO),
  \end{align*}
  where we used the identity $\Phi_\star+\Phi^\CO-\Phi_\star^\CO = \Phi-\Phi^\NC+\Phi_\star^\NC$ in the second step. Using the stability, ellipticity, and
  Young's inequality $ab \leq a^2/(4\varepsilon) +\varepsilon b^2$, we obtain
  \begin{align*}
    2a_\star(\phi-\Phi,\Phi^\NC - \Phi_\star^\NC) &\leq 2a(\phi-\Phi,\phi-\Phi)^{1/2}a_{\star}(\Phi^\NC-\Phi_\star^\NC,\Phi^\NC-\Phi_\star^\NC)^{1/2}\\
    &\leq \varepsilon a(\phi-\Phi,\phi-\Phi) + \varepsilon^{-1}\c{norm}^2 \enorm{\Phi^\NC-\Phi_\star^\NC}_\star^2,
  \end{align*}
  as well as
  \begin{align*}
    \frac{\c{norm}^{-2}}{2}\enorm{\Phi_\star-\Phi}_\star^2 - \c{norm}^{-2}\enorm{\Phi_\star^\NC-\Phi^\NC}_\star^2 \leq
    \c{norm}^{-2} \enorm{\Phi_\star^\CO-\Phi^\CO}_\star^2 &\leq a_\star(\Phi_\star^\CO-\Phi^\CO,\Phi_\star^\CO-\Phi^\CO).
  \end{align*}
  Finally, the estimate
  \begin{align*}
    a_\star(\Phi^\NC-\Phi_\star^\NC,\Phi^\NC-\Phi_\star^\NC) &\leq \c{norm}^2 \enorm{\Phi^\NC-\Phi_\star^\NC}_\star^2
  \end{align*}
  concludes the proof.
\end{proof}

\subsection{Quasi-interpolation operators in $\wilde H^{-1/2}(\Gamma)$}\label{section:interpolation}
Lemma~\ref{lem:orth} will be the basis for the analysis of the a posteriori error estimators in Section~\ref{section:aposteriori},
such that terms of the form $\enorm{\Phi-\Phi^\CO}$ will emerge. Those terms are (in principle) computable. However,
they involve conforming approximations $\Phi^\CO$, which we don't want to compute, and hence we need to find a substitute involving only $\Phi$.
This will be done in Corollary~\ref{cor:ncon}, where we will estimate the nonconformity of a function $\Phi$ by its jumps over edges.
The proof of this corollary will be based on results of the present section, the aim of which is to
provide an interpolation operator to approximate the conforming part $\Phi^0$ of a given function $\Phi$.
We will use the well-known interpolation operator $I_\mesh$ by Cl\'ement~\cite{clement,scottzhang}, and provide
approximation results in the space $\wilde \bH^{-1/2}(\Gamma)$.
For a function $v\in L_2(\Gamma)$, this operator is defined as
\begin{align}\label{eq:I}
    &I_\mesh v := \sum_{z\in\nodes_\mesh} \psi(z) \varphi_z,
\end{align}
where $\varphi_z$ is the nodal basis function of $\FE_\mesh$ associated with the node $z\in\nodes_\mesh$.
The function $\psi\in\FE_\mesh|_{\omega_z}$ is such that
\begin{align*}
  \int_{\omega_z}(v-\psi) \varphi  = 0 \quad\text{ for all } \varphi\in\FE_\mesh|_{\omega_z}
\end{align*}
see also~\cite[Lemma 6.6]{bonito:nochetto}.
In addition, we denote by $\Pi_\mesh$ the $L_2(\Gamma)$ orthogonal projection onto the space of piecewise constants $[\PP^0(\mesh)]^2$.
The well-known properties of the operator $I_\mesh$ are collected
in the following lemma. We again refer to~\cite[Lemma 6.6]{bonito:nochetto} for a proof.
\begin{lemma}\label{lem:I}
  Let $\mesh$ be a refinement of $\mesh_0$. Then, there exists a constant $\setc{I}$ which depends only on $\mesh_0$ such that
  \begin{align}\label{eq:I:stab}
    \norm{I_\mesh \varphi}{L_2(\Gamma)} \leq \c{I}\norm{\varphi}{L_2(\Gamma)}\quad \text{ and }\quad \norm{I_\mesh \varphi}{H^1(\Gamma)} \leq \c{I}\norm{\varphi}{H^1(\Gamma)},
  \end{align}
  and such that for all $\el\in\mesh$, for all $\varphi\in H^1_0(\Gamma)$, and for all $\Phi\in\CR_\mesh$, it holds that
  \begin{subequations}\label{eq:I:apx}
    \begin{align}
      \norm{\varphi-I_\mesh \varphi}{L_2(\el)} &\leq \c{I} \norm{h_\mesh \nabla \varphi}{\bL_2(\elpatch)}\label{eq:I:apx:vol},\\
      \norm{\Phi-I_\mesh \Phi}{L_2(\el)} &\leq \c{I} \norm{h_\mesh^{1/2}\jump{\Phi}}{L_2(\edges_\elpatch)}\label{eq:I:sides:1},\\
      \norm{\nabla_\mesh(\Phi-I_\mesh \Phi)}{\bL_2(\el)} &\leq \c{I} \norm{h_\mesh^{-1/2}\jump{\Phi}}{L_2(\edges_\elpatch)}.\label{eq:I:sides:2}
    \end{align}
  \end{subequations}
  \hfill$\qed$
\end{lemma}
For our purposes, we need to analyze the properties of $I_\mesh$ in the space $\wilde\bH^{-1/2}(\Gamma)$.
To do so, we will use integration by parts piecewise. The resulting integrals over the skeleton $\edges_\mesh$
will be bounded with the aid of the following auxiliary result.
\begin{lemma}\label{lem:poinc}
  Let $\mesh$ be a refinement of $\mesh_0$ with the set of edges $\edges_\mesh$. Then, there is a constant $\setc{poinc}$ which depends only on $\mesh_0$
  such that for any choice of functions $\Phi\in\CR_\mesh$ and $\V\in\left[ \FE_\mesh \right]^2$, it holds that
  \begin{align}\label{eq:poinc}
    \int_{\edges_\mesh}\jump{\Phi}\avg{\V} \leq \c{poinc} \norm{\jump{\Phi}}{L_2(\edges_\mesh)} \norm{\V}{\bH^{1/2}(\Gamma)}.
  \end{align}
  Furthermore, if $\wat\mesh$ is the uniform refinement of $\mesh$ and $\wat\Phi\in\FE_{\wat\mesh}$, it holds that
  \begin{align}\label{eq:poinc:2}
    \int_{\edges_{\wat\mesh}}\jump{\wat\Phi}\avg{\V} \leq \c{poinc} \norm{h_\mesh^{1/2}(1-\Pi_\mesh)\nabla_{\wat\mesh}\wat\Phi}{\bL_2(\Gamma)} \norm{\V}{\bH^{1/2}(\Gamma)}.
  \end{align}
\end{lemma}
\begin{proof}
  For every edge $\ed\in\edges_\mesh$, we use an affine map to transfer the edge patch $\omega_\ed$ to a reference configuration $\overline{\omega_\ed}$.
  As we emphasized in Section~\ref{section:meshes}, the number of this reference configurations is bounded uniformly, which permits
  us to use scaling arguments. Now we choose $\bfc_\ed\in\R^2$ such that
  \begin{align*}
    \norm{\overline\V - \bfc_\ed}{\bL_2(\overline{\ed})} \lesssim \snorm{\overline\V}{\bH_{\rm slo}^{1/2}(\overline{\omega_\ed})},
  \end{align*}
  which is possible since $\overline\V$ is an element of a finite dimensional space. Here, the index ${\rm slo}$ indicates that the norm is defined according to Sobolev-Slobodeckij.
  Mapping both sides back to the physical domain yields
  \begin{align}\label{lem:poinc:eq:1}
    \norm{\V - \bfc_\ed}{\bL_2(\ed)} \lesssim \snorm{\V}{\bH_{\rm slo}^{1/2}(\omega_\ed)} \leq \norm{\V}{\bH_{\rm slo}^{1/2}(\omega_\ed)}.
  \end{align}
  As $\Phi$ is a Crouzeix-Raviart function, its jump $\jump{\Phi}$ has vanishing integral mean on every edge $\ed\in\edges_\mesh$, and hence,
  using the Cauchy-Schwarz inequality, we obtain with~\eqref{lem:poinc:eq:1}
  \begin{align*}
    \begin{split}
      \int_{\edges_\mesh}\jump{\Phi}\avg{\V} = \sum_{\ed\in\edges_\mesh} \int_{\ed}\jump{\Phi}\avg{\V-\bfc_\ed}
      &\leq \left(\sum_{\ed\in\edges_\mesh} \norm{\jump{\Phi}}{L_2(\ed)}^2\right)^{1/2} \left(\sum_{\ed\in\edges_\mesh}\norm{\V - \bfc_\ed}{\bL_2(\ed)}^2\right)^{1/2}\\
      &\leq \left(\sum_{\ed\in\edges_\mesh} \norm{\jump{\Phi}}{L_2(\ed)}^2\right)^{1/2} \left(\sum_{\ed\in\edges_\mesh}\norm{\V}{\bH^{1/2}_{\rm slo}(\omega_\ed)}^2 \right)^{1/2}.
      \end{split}
  \end{align*}
  Locally, only three patches $\omega_\ed$ overlap, and the fact that the norms $\bH_{\rm slo}^{1/2}(\Gamma)$
  and $\bH^{1/2}(\Gamma)$ are equivalent finally concludes the proof of~\eqref{eq:poinc}.
  Now we prove~\eqref{eq:poinc:2}. We start at~\eqref{eq:poinc}, this time with $\wat\mesh$ instead of $\mesh$, to obtain
  \begin{align*}
    \int_{\edges_{\wat\mesh}}\jump{\wat\Phi}\avg{\V} \lesssim \left(\sum_{\ed\in\edges_{\wat\mesh}} \norm{\jump{\wat\Phi}}{L_2(\ed)}^2\right)^{1/2} \norm{\V}{\bH^{1/2}(\Gamma)}.
  \end{align*}
  Now we split the $L_2$ norm of the jump $\jump{\wat\Phi}$ over the skeleton $\edges_{\wat\mesh}$ into the contributions on the skeleton $\edges_\mesh$ and the rest,
  which we write sloppy as $\edges_{\wat\mesh}\setminus\edges_\mesh$. Then,
  \begin{align}\label{lem:poinc:eq:4}
    \sum_{\ed\in\edges_{\wat\mesh}} \norm{\jump{\wat\Phi}}{L_2(\ed)}^2 = \sum_{\ed\in\edges_{\mesh}}\norm{\jump{\wat\Phi}}{L_2(\ed)}^2
    + \sum_{\ed\in\edges_{\wat\mesh}\setminus\edges_\mesh}\norm{\jump{\wat\Phi}}{L_2(\ed)}^2.
  \end{align}
  We claim that there is a constant $C>0$, independent of $\edges_{\wat\mesh}$ and $\wat\Phi$ such that
  \begin{align*}
    \norm{\jump{\wat\Phi}}{L_2(\ed)} &\leq C h_\ed^{1/2}\norm{(1-\Pi_\mesh)\nabla_{\wat\mesh}\wat\Phi}{\bL_2(\omega_\ed)}\quad \text{ if } \ed\in\edges_{\mesh},\\
    \norm{\jump{\wat\Phi}}{L_2(\ed)} &\leq C h_\ed^{1/2}\norm{(1-\Pi_\mesh)\nabla_{\wat\mesh}\wat\Phi}{\bL_2(\el)}
    \quad\text{ if } \ed\in\edges_{\wat\mesh}\setminus\edges_\mesh \text{ with } \ed\subset\el\in\mesh.
  \end{align*}
  Both sides define seminorms, and the left one vanishes when the right one does. Hence, the bounded dimension of the underlying space and a scaling argument prove the claim.
  Using the last two estimates in~\eqref{lem:poinc:eq:4} shows~\eqref{eq:poinc:2}.
\end{proof}
\begin{lemma}\label{lem:I:neg}
  In addition to Lemma~\ref{lem:I}, we have the following estimates, where $\wat\mesh$ denotes the uniform refinement of $\mesh$: For
  $\Phi\in\CR_\mesh$ and $\wat\Phi\in\CR_{\wat\mesh}$, it holds that
  \begin{subequations}\label{eq:I:neg}
    \begin{align}
      \norm{\nabla_\mesh(1-I_\mesh)\Phi}{\wilde\bH^{-1/2}(\Gamma)} &\leq \c{I} \norm{h_\mesh\jump{\Phi}'}{L_2(\edges_\mesh)}\label{eq:I:neg:1},\\
      \norm{\nabla_{\wat\mesh}(1-I_\mesh)\wat\Phi}{\wilde\bH^{-1/2}(\Gamma)} &\leq \c{I} \norm{h_\mesh^{1/2}(1-\Pi_\mesh)\nabla_{\wat\mesh}\wat\Phi}{\bL_2(\Gamma)}.\label{eq:I:neg:2}
    \end{align}
  \end{subequations}
\end{lemma}
\begin{proof}
  We will use estimates~\eqref{eq:I:stab} and~\eqref{eq:I:apx} to prove this lemma.
  First, if we denote by $I_\mesh\vv$ the component-wise action of $I_\mesh$ to $\vv\in\bH^{1/2}(\Gamma)$, we integrate by parts piecewise to obtain
  \begin{align*}
    \dual{\nabla_\mesh(1 - I_\mesh)\Phi}{I_\mesh\vv} = -\dual{(1-I_\mesh)\Phi}{\div I_\mesh \vv} + \sum_{\el\in\mesh} \int_{\edges_\el} (\Phi-I_\mesh\Phi) I_\mesh\vv\cdot\n_\el.
  \end{align*}
  As $\jump{I_\mesh\Phi}$ vanishes due to the continuity of $I_\mesh\Phi$, the second term on the right-hand side can be written as
  \begin{align*}
    \sum_{\el\in\mesh} \int_{\edges_\el} (\Phi-I_\mesh\Phi) I_\mesh\vv\cdot\n_\el &=
    \int_{\edges_\mesh}\jump{\Phi-I_\mesh\Phi}\avg{I_\mesh\vv} +
    \int_{\edges_\mesh\setminus\partial\Gamma}\avg{\Phi-I_\mesh\Phi}\jump{I_\mesh\vv}\\
    &= \int_{\edges_\mesh}\jump{\Phi}\avg{I_\mesh\vv}.
  \end{align*}
  We conclude that, for any $\vv\in\bH^{1/2}(\Gamma)$,
  \begin{align}\label{eq:lem:I:1}
    \begin{split}
      \dual{\nabla_\mesh(1-I_\mesh)\Phi}{\vv} &= \dual{\nabla_\mesh(1-I_\mesh)\Phi}{\vv-I_\mesh\vv} - \dual{(1-I_\mesh)\Phi}{\div I_\mesh\vv}\\
      &\qquad+ \int_{\edges_\mesh}\jump{\Phi}\avg{I_\mesh\vv}.
    \end{split}
  \end{align}
  We bound the terms on the right-hand side separately.
  Taking into account~\eqref{eq:I:sides:2}, the first term on the right-hand side of~\eqref{eq:lem:I:1} can be estimated by
  \begin{align}\label{eq:lem:I:4}
    \begin{split}
      \dual{\nabla_\mesh(1-I_\mesh)\Phi}{\vv-I_\mesh\vv} &\leq \sum_{\el\in\mesh}\norm{\nabla_\mesh(1-I_\mesh)\Phi}{\bL_2(\el)} \norm{\vv-I_\mesh\vv}{\bL_2(\el)}\\
      &\lesssim \sum_{\el\in\mesh} h_\mesh|_\el^{-1/2}\norm{\jump{\Phi}}{L_2(\edges_\elpatch)} \norm{\vv-I_\mesh\vv}{\bL_2(\el)}\\
      &\leq \norm{\jump{\Phi}}{L_2(\edges_\mesh)} \norm{h_\mesh^{-1/2}(\vv-I_\mesh\vv)}{\bL_2(\Gamma)}.
    \end{split}
  \end{align}
  Now, it holds that $\norm{h_\mesh^{-1/2}(\vv-I_\mesh\vv)}{\bL_2(\Gamma)} \lesssim \norm{\vv}{\bH^{1/2}(\Gamma)}$, which follows from interpolation of the estimates
  \begin{align*}
    \norm{\vv-I_\mesh\vv}{\bL_2(\Gamma)}\lesssim\norm{\vv}{\bL_2(\Gamma)} \quad \text{ and } \quad \norm{h_\mesh^{-1}(\vv-I_\mesh\vv)}{\bL_2(\Gamma)} \lesssim \norm{\vv}{\bH^1(\Gamma)},
  \end{align*}
  which themselves can be derived summing~\eqref{eq:I:stab} and~\eqref{eq:I:apx:vol} over the elements of the mesh. We conclude that
  \begin{align}\label{eq:lem:I:2}
    \dual{\nabla_\mesh(1-I_\mesh)\Phi}{\vv-I_\mesh\vv} \lesssim \norm{\jump{\Phi}}{L_2(\edges_\mesh)}\norm{\vv}{\bH^{1/2}(\Gamma)}.
  \end{align}
  The second contribution on the right-hand side of~\eqref{eq:lem:I:1} can be bounded by using~\eqref{eq:I:sides:1} via
  \begin{align}\label{eq:lem:I:3}
    \begin{split}
    \dual{\Phi-I_\mesh\Phi}{\div I_\mesh\vv} &\leq \sum_{\el\in\mesh}\norm{\Phi-I_\mesh\Phi}{L_2(\el)}\norm{\div I_\mesh\vv}{L_2(\el)}\\
    &\lesssim\norm{\jump{\Phi}}{L_2(\edges_\mesh)}\norm{h_\mesh^{1/2}\div I_\mesh\vv}{L_2(\Gamma)}\\
    &\lesssim\norm{\jump{\Phi}}{L_2(\edges_\mesh)}\norm{\vv}{\bH^{1/2}(\Gamma)}.
    \end{split}
  \end{align}
  In the last step we used an inverse estimate, cf.~\cite[Proposition 3.1]{cp07} and the recent extension~\cite[Proposition 5]{affkp13-A}, and the fact that $I_\mesh$ is bounded in $\bH^{1/2}(\Gamma)$, which again follows by interpolation,
  this time using the estimates~\eqref{eq:I:stab}. The third part on the right-hand side of~\eqref{eq:lem:I:1} can be bounded by Lemma~\ref{lem:poinc} and
  the $\bH^{1/2}(\Gamma)$-boundedness of $I_\mesh$ via
  \begin{align}\label{eq:lem:I:6}
    \int_{\edges_\mesh}\jump{\Phi}\avg{I_\mesh\vv} \lesssim \norm{\jump{\Phi}}{L_2(\edges_\mesh)} \norm{\vv}{\bH^{1/2}(\Gamma)}.
  \end{align}
  From the identity~\eqref{eq:lem:I:1} we conclude, using~\eqref{eq:lem:I:2},~\eqref{eq:lem:I:3}, and~\eqref{eq:lem:I:6}, that
  \begin{align*}
    \norm{\nabla_\mesh(1-I_\mesh)\Phi}{\wilde\bH^{-1/2}(\Gamma)} &= \sup_{\norm{\vv}{\bH^{1/2}(\Gamma)}=1} \dual{\nabla_\mesh(1-I_\mesh)\Phi}{\vv}
    \lesssim\norm{\jump{\Phi}}{L_2(\edges_\mesh)}.
  \end{align*}
  From this,~\eqref{eq:I:neg:1} follows from a Poincar\'e inequality, which may be used since $\Phi\in\CR_\mesh$ implies that the jump $\jump{\Phi}$ vanishes at the midpoint
  of every element.
  
  To prove~\eqref{eq:I:neg:2}, we again use integration by parts piecewise and conclude as before
  \begin{align}\label{eq:lem:I:8}
    \begin{split}
      \dual{\nabla_{\wat\mesh}(1-I_\mesh)\wat\Phi}{\vv} &= \dual{\nabla_{\wat\mesh}(1-I_\mesh)\wat\Phi}{\vv-I_\mesh\vv} - \dual{(1-I_\mesh)\wat\Phi}{\div I_\mesh\vv},\\
      &\qquad+ \int_{\edges_{\wat\mesh}}\jump{\wat\Phi}\avg{I_\mesh\vv}.
    \end{split}
  \end{align}
  The first and second term can be bounded as in~\eqref{eq:lem:I:4} and~\eqref{eq:lem:I:3},
  this time using the local estimates
  \begin{align*}
    \norm{\nabla_{\wat\mesh}(1-I_\mesh)\wat\Phi}{\bL_2(\el)} \leq C \norm{(1-\Pi_\mesh)\nabla_{\wat\mesh}\wat\Phi}{\bL_2(\omega_\el)}\\
    \norm{(1-I_\mesh)\wat\Phi}{L_2(\el)} \leq C h_\mesh|_\el\norm{(1-\Pi_\mesh)\nabla_{\wat\mesh}\wat\Phi}{\bL_2(\omega_\el)},
  \end{align*}
  which follow from a scaling argument and norm equivalence in finite dimensional spaces.
  The last term in~\eqref{eq:lem:I:8} can be bounded by~\eqref{eq:poinc:2} of Lemma~\ref{lem:poinc}.
\end{proof}
We will also need the following boundedness result for $I_\mesh$.
\begin{lemma}\label{lem:I:stab}
  In addition to Lemma~\ref{lem:I}, we have the following estimate, where $\wat\mesh$ denotes the uniform refinement of $\mesh$: For
  $\Phi\in\CR_\mesh$ and $\wat\Phi\in\CR_{\wat\mesh}$, it holds that
  \begin{equation} 
      \norm{\nabla_{\mesh}I_\mesh \widehat\Phi}{\wilde\bH^{-1/2}(\Gamma)} \leq \c{I} \norm{\nabla_{\widehat\mesh}\widehat \Phi}{\wilde\bH^{-1/2}(\Gamma)}.\label{eq:I:disc:stab:1}
  \end{equation}
\end{lemma}
\begin{proof}
  To prove~\eqref{eq:I:disc:stab:1}, we first observe that due to the local $L_2$ boundedness of $(1-\Pi_\mesh)$ and the inverse estimate~\cite[Thm. 3.6]{ghs},
  we have
  \begin{align*}
    \norm{h_\mesh^{1/2}(1-\Pi_\mesh)\nabla_{\wat\mesh}\wat\Phi}{\bL_2(\Gamma)} \leq \norm{h_\mesh^{1/2}\nabla_{\wat\mesh}\wat\Phi}{\bL_2(\Gamma)}
    \lesssim \norm{\nabla_{\wat\mesh}\wat\Phi}{\wilde\bH^{-1/2}(\Gamma)}.
  \end{align*}
  Hence, the triangle inequality and~\eqref{eq:I:neg:2} show
  \begin{align*}
    \norm{\nabla_{\wat\mesh}I_\mesh\wat\Phi}{\wilde\bH^{-1/2}(\Gamma)} \leq \norm{\nabla_{\wat\mesh}\wat\Phi}{\wilde\bH^{-1/2}(\Gamma)} +
    \norm{\nabla_{\wat\mesh}(1-I_\mesh)\wat\Phi}{\wilde\bH^{-1/2}(\Gamma)} \lesssim \norm{\nabla_{\wat\mesh}\wat\Phi}{\wilde\bH^{-1/2}(\Gamma)}.
  \end{align*}
\end{proof}
\section{A posteriori error estimation and adaptive algorithm}\label{section:aposteriori}
In this section, we introduce different error estimators, and show their reliability and efficiency. In Section~\ref{section:global:est}, we introduce
global error estimators, that is, the employed (non-integer) norm is nonlocal and therefore does not provide information for local mesh-refinement.
In Section~\ref{section:local:est}, we pass over to weighted (integer) norms, which are local and can therefore be employed in an adaptive algorithm,
which will be introduced in Section~\ref{section:algorithm}.
In order to estimate the nonconformity of a function in terms of the function itself, we will use the
results of Sections~\ref{section:conform} and~\ref{section:interpolation}.
\begin{corollary}\label{cor:ncon}
  Denote by $\mesh$ a refinement of $\mesh_0$.
  Let $\Phi\in\CR_\mesh$ be the Galerkin solution~\eqref{eq:galerkin}. Then, there is a constant $\setc{nc}>0$ which depends only on $\mesh_0$ such that
  \begin{align*}
    \enorm{\Phi^\perp}_\mesh = \enorm{\Phi-\Phi^\CO}_\mesh \leq \c{nc} \norm{h_\mesh\jump{\Phi}'}{L_2(\edges_\mesh)}.
  \end{align*}
\end{corollary}
\begin{proof}
  This follows easily by using the fact that $\Phi-\Phi^\CO$ is $a_\mesh$-orthogonal to $\FE_\mesh$ and employing~\eqref{eq:I:neg:1}.
\end{proof}
\subsection{Global error estimators}\label{section:global:est}
Let $\Phi\in\CR_\mesh$ and $\widehat\Phi\in\CR_{\wat\mesh}$ be Galerkin solutions~\eqref{eq:galerkin}, where $\widehat\mesh$ is a uniform refinement of $\mesh$.
We introduce estimators on the mesh $\mesh$ by
\begin{align*}
  \eta_\mesh &:= \enorm{\wat\Phi-\Phi}_{\wat\mesh} = \norm{\scurl_{\wat\mesh}(\wat\Phi-\Phi)}{\wilde\bH^{-1/2}(\Gamma)},\\
  \wilde\eta_\mesh &:= \enorm{\wat\Phi-I_\mesh\wat\Phi}_{\wat\mesh} = \norm{\scurl_{\wat\mesh}(\wat\Phi-I_\mesh \wat\Phi)}{\wilde\bH^{-1/2}(\Gamma)}.
\end{align*}
The existing derivations of $h-h/2$ error estimators, e.g.~\cite{effp09,fp08}, focus on conforming methods and rely mostly on the Galerkin orthogonality~\eqref{eq:orth}.
Contrary, we have the weaker partial orthogonality of Lemma~\ref{lem:orth}, where additional terms arise (what we called nonconformity error) which account for the nonconformity.
In Corollary~\ref{cor:ncon}, we showed that these terms can be bounded by the inter-element jumps of $\Phi$, i.e., by
\begin{align*}
  \ncon_\mesh := \norm{h_\mesh\jump{\Phi}'}{L_2(\edges_\mesh)}.
\end{align*}

Consequently, we have that $\eta_\mesh$ and $\wilde\eta_\mesh$ are equivalent up to $\ncon_\mesh$.
\begin{lemma}\label{lem:eta:etatilde}
  Let $\mesh$ be a refinement of $\mesh_0$. Then, there is a constant $\setc{eta:etatilde}>0$ which depends only on $\mesh_0$ such that
  \begin{align*}
    \c{eta:etatilde}^{-1}\enorm{\wat\Phi-\Phi}_{\wat\mesh} \leq \enorm{\wat\Phi-I_\mesh\wat\Phi}_{\wat\mesh} + \ncon_\mesh \quad\text{ and }
    \quad \c{eta:etatilde}^{-1}\enorm{\wat\Phi-I_\mesh\wat\Phi}_{\wat\mesh} \leq \enorm{\wat\Phi-\Phi}_{\wat\mesh} + \ncon_\mesh.
  \end{align*}
\end{lemma}
\begin{proof}
  As $\wat\Phi-\Phi$ is orthogonal to $\FE_\mesh$ in $a_{\wat\mesh}$, we conclude
  \begin{align*}
    \enorm{\wat\Phi-\Phi}_{\wat\mesh} \lesssim \enorm{\wat\Phi-\Phi+\Phi^0-I_\mesh\wat\Phi} \leq \enorm{\wat\Phi-I_\mesh\wat\Phi}_{\wat\mesh} + \enorm{\Phi-\Phi^0}_{\mesh},
  \end{align*}
  and the last term can be bounded by $\ncon_\mesh$ by Corollary~\ref{cor:ncon}.
  To see the second estimate,
  we use the projection property and boundedness~\eqref{eq:I:disc:stab:1} of $I_\mesh$ to see that
  \begin{align*}
    \enorm{\wat\Phi-I_\mesh\wat\Phi}_{\wat\mesh} \lesssim \enorm{\wat\Phi-\Phi^\CO}_{\wat\mesh} \leq \enorm{\wat\Phi-\Phi}_{\wat\mesh} + \enorm{\Phi-\Phi^\CO}_{\mesh},
  \end{align*}
  which shows the desired estimate.
\end{proof}
In a next step, we show the efficiency and reliability of $\eta_\mesh$. For the reliability, we assume that a saturation
assumption for the \textit{conforming} approximations holds true.
\begin{theorem}\label{thm:eta:eff:rel}
  Let $\mesh$ be a refinement of $\mesh_0$. Then, there is a constant $\setc{eff}>0$ such that $\eta_\mesh = \enorm{\wat\Phi-\Phi}_{\wat\mesh}$ is efficient up to the nonconformity error, i.e.,
  \begin{align}\label{eq:eta:eff}
    \c{eff}^{-1}\enorm{\wat\Phi-\Phi}_{\wat\mesh} \leq \enorm{\phi-\Phi}_{\mesh} + \ncon_\mesh + \ncon_{\wat\mesh}.
  \end{align}
  Furthermore, assume that there is a constant $\setc{sat}\in(0,1)$ such that the saturation assumption for the conforming approximations
  \begin{align}\label{eq:sat}
    a_{\wat\mesh}(\phi-\wat\Phi^\CO,\phi-\wat\Phi^\CO) \leq \c{sat} a_{\mesh}(\phi-\Phi^\CO,\phi-\Phi^\CO)
  \end{align}
  holds true. Then, there is a constant $\setc{rel}>0$ such that $\eta_\mesh = \enorm{\wat\Phi-\Phi}_{\wat\mesh}$ is reliable up to $\ncon_\mesh+\ncon_{\wat\mesh}$, i.e.,
  \begin{align}\label{eq:eta:rel}
    \c{rel}^{-1}\enorm{\phi-\Phi}_{\mesh} \leq \enorm{\wat\Phi-\Phi}_{\wat\mesh} + \ncon_\mesh + \ncon_{\wat\mesh}
  \end{align}
  holds true.
\end{theorem}
\begin{proof}
  Efficiency~\eqref{eq:eta:eff} follows immediately from Lemma~\ref{lem:orth} by setting $\mesh_\star := \widehat\mesh$ and Corollary~\ref{cor:ncon}.

  To show reliability~\eqref{eq:eta:rel}, we first note that the triangle inequality and ellipticity give
  \begin{align*}
    \enorm{\phi-\Phi} \lesssim a(\phi-\Phi^\CO,\phi-\Phi^\CO) + \ncon_\mesh.
  \end{align*}
  Now, due to the conforming orthogonality and the saturation assumption~\eqref{eq:sat},
  \begin{align*}
    (1-\c{sat}) a(\phi-\Phi^\CO,\phi-\Phi^\CO) &\leq a(\Phi^\CO-\wat\Phi^\CO,\Phi^\CO-\wat\Phi^\CO) \lesssim \enorm{\Phi^\CO-\wat\Phi^\CO}_{\wat\mesh}\\
    &\leq \enorm{\wat\Phi-\Phi}_{\wat\mesh} + \enorm{\Phi-\Phi^\CO} + \enorm{\wat\Phi-\wat\Phi^\CO}\\
    &\lesssim \enorm{\wat\Phi-\Phi}_{\wat\mesh} + \ncon_\mesh + \ncon_{\wat\mesh},
  \end{align*}
  where we used the triangle inequality and Corollary~\ref{cor:ncon}.
\end{proof}
\begin{remark}
  In finite element methods, the saturation assumption~\eqref{eq:sat} is verified for the Poisson problem $-\Delta u = f$. In fact, in~\cite{dn02} it is shown that
  \begin{align*}
    \enorm{\phi-\widehat\Phi^0_\ell} \leq \c{sat} \enorm{\phi-\Phi^0_\ell} + \osc_\ell,
  \end{align*}
  where $\osc_\ell$ is a measure for the resolution of $f$ on the mesh $\mesh_\ell$. Hence, small data oscillation implies the saturation assumption. However, the saturation
  assumption~\eqref{eq:sat} is not proven for BEM. To the best of the our knowledge, the only contributions are~\cite{affkp13,eh06}. In~\cite{affkp13}, it is shown that
  for $2D$-BEM for the weakly singular integral equation,
  there is a $k\in\mathbb{N}$ and $\c{sat}<1$ which depend only on $\mesh_0$ and $\Gamma$, such that with $k$ uniform refinements of $\mesh_\ell$, which we denote by $\mesh_{\ell(k)}$, there holds
  \begin{align*}
    \enorm{\phi-\Phi_\ell} \leq \c{sat}\enorm{\phi-\Phi_{\ell(k)}} + \osc_\ell
  \end{align*}
  with $\osc_\ell$ being a term of higher order than the others. In~\cite{eh06}, the saturation assumption is analyzed for an edge singularity on a plane square-shaped domain, and uniform as well as
  graded meshes are considered.
\end{remark}
\subsection{Localized error estimators}\label{section:local:est}
The a posteriori estimators of Section~\ref{section:global:est} use the $\wilde\bH^{-1/2}(\Gamma)$-norm, which is 
hard to compute and nonlocal. In order to provide a posteriori error estimators which can be split into element-wise indicators,
we will use a weighted $\bL_2$-norm. We introduce the localized estimators
\begin{align*}
  \mu_\mesh &:= \norm{h_\mesh^{1/2}\scurl_{\wat\mesh}(\wat\Phi-\Phi)}{\bL_2(\Gamma)},\\
  \wilde\mu_\mesh &:= \norm{h_\mesh^{1/2}(\scurl_{\wat\mesh}\wat\Phi - \Pi_\mesh \scurl_{\wat\mesh}\wat\Phi)}{\bL_2(\Gamma)}.
\end{align*}
Then we have the following result.
\begin{lemma}\label{lem:mu}
  There holds
\begin{subequations}
  \begin{align}\label{eq:est:equivalence:1}
    \norm{h_\mesh^{1/2}(\scurl_{\wat\mesh}\wat\Phi - \Pi_\mesh \scurl_{\wat\mesh}\wat\Phi)}{\bL_2(\Gamma)}
    \leq \norm{h_\mesh^{1/2}\scurl_{\wat\mesh}(\wat\Phi-\Phi)}{\bL_2(\Gamma)}
    \lesssim \enorm{\wat\Phi-\Phi}_{\wat\mesh}
  \end{align}
    and
  \begin{align}\label{eq:est:equivalence:2}
    \enorm{\wat\Phi-I_\mesh\wat\Phi}_{\wat\mesh} \lesssim \norm{h_\mesh^{1/2}(\scurl_{\wat\mesh}\wat\Phi - \Pi_\mesh \scurl_{\wat\mesh}\wat\Phi)}{\bL_2(\Gamma)}.
  \end{align}
\end{subequations}
  In particular, all estimators are equivalent up to $\ncon_\mesh + \ncon_{\wat\mesh}$, and for
  $\tau\in\left\{ \eta_\mesh,\wilde\eta_\mesh,\mu_\mesh,\wilde\mu_\mesh \right\}$,
  the estimator $\tau$ is reliable and efficient up to $\ncon_\mesh + \ncon_{\wat\mesh}$, i.e.,
  \begin{align*}
    \enorm{\phi-\Phi} &\lesssim \tau + \ncon_\mesh + \ncon_{\wat\mesh},\\
    \tau &\lesssim \enorm{\phi-\Phi} + \ncon_\mesh + \ncon_{\wat\mesh}.
  \end{align*}
\end{lemma}
\begin{proof}
  As in~\cite{fp08}, the first estimate in~\eqref{eq:est:equivalence:1}
  follows from the best approximation property of $\Pi_\mesh$, while the second one follows
  from the inverse inequality~\cite[Theorem 3.6]{ghs}.
  The first estimate in~\eqref{eq:est:equivalence:2} is estimate~\eqref{eq:I:neg:2} from Lemma~\ref{lem:I:neg}.
  Now, since
  \begin{align*}
    \eta_\mesh = \enorm{\wat\Phi-\Phi}_{\wat\mesh} \quad \text{ and } \quad
    \wilde\eta_\mesh = \enorm{\wat\Phi-I_\mesh\wat\Phi}_{\wat\mesh}
  \end{align*}
  are equivalent up to $\ncon_\mesh + \ncon_{\wat\mesh}$ according to Lemma~\ref{lem:eta:etatilde},
  all estimators are equivalent up to $\ncon_\mesh + \ncon_{\wat\mesh}$ as well. As $\eta_\mesh$ is efficient and reliable up to $\ncon_\mesh + \ncon_{\wat\mesh}$ according
  to Theorem~\ref{thm:eta:eff:rel}, this is also true for the three other estimators.
\end{proof}
\subsection{Statement of the adaptive algorithm}\label{section:algorithm}
We now introduce the adaptive algorithm. As error indicators on a mesh $\mesh_\ell$, we use the element-wise quantities
\begin{align*}
  \varrho_\ell(\el)^2 := \norm{h_\ell^{1/2}(1-\Pi_\ell)\scurl_{\wat\mesh_\ell}\wat\Phi_\ell}{\bL_2(\el)}^2
  + \norm{h_\ell\jump{\Phi_\ell}}{H^1(\edges_\ell(\el))}^2
  + \norm{\widehat h_\ell \jump{\wat\Phi_\ell}}{H^1(\wat\edges_\ell(\el))}^2.
\end{align*}
For a subset $\MM_\ell \subset \mesh_\ell$, we write $\varrho_\ell(\MM_\ell)^2 = \sum_{\el\in\MM_\ell}\varrho_\ell(\el)^2$, and
we use the abbreviation $\varrho_\ell := \varrho_\ell(\mesh_\ell)$. Hence,
\begin{align*}
  \varrho_\ell^2 = \wilde\mu_\ell^2 + \rho_\ell^2 + \wat\rho_\ell^2
\end{align*}
is a reliable error estimator according to Lemma~\ref{lem:mu}.
The adaptive algorithm now reads as follows.
\begin{algorithm}\label{algorithm}
Input: Initial mesh $\mesh_0$, parameter $\theta \in (0,1)$, counter $\ell := 0$.
\begin{itemize}
  \item[(i)] Obtain $\widehat\mesh_\ell$ by uniform bisec(3)-refinement of $\mesh_\ell$, see Figure~\ref{fig:nvb}.
  \item[(ii)] Compute solutions $\Phi_\ell$ and $\widehat\Phi_\ell$ of~\eqref{eq:galerkin} with respect to $\mesh_\ell$ and $\widehat\mesh_\ell$.
  \item[(iii)] Compute refinement indicators $\varrho_\ell(\el)$ for all $\el \in \mesh_\ell$.
  \item[(iv)] Choose a set $\MM_\ell\subseteq\mesh_\ell$ with minimal cardinality such that
    \begin{align}\label{eq:doerfler}
      \sum_{\el \in \MM_\ell} \varrho_\ell(\el)^2 \geq \theta\sum_{\el \in \mesh_\ell} \varrho_\ell(\el)^2.
    \end{align}
  \item[(v)] Refine mesh $\mesh_\ell$ according to Algorithm NVB and obtain $\mesh_{\ell+1}$.
  \item[(vi)] Update counter $\ell := \ell+1$ and goto (i).
\end{itemize}
\end{algorithm}
\section{Numerical experiments}\label{section:numerics}
In this section we present numerical experiments for two different problems. The exact solution $\phi$ of the first experiment
will be \textit{smooth} in the sense that uniform and adaptive mesh refinement yield the same rate of convergence.
Still, $\phi$ exhibits singularities which stem from the geometric setting (i.e., polygonal boundary).
As we emphasized in the introduction, it is a peculiarity of Crouzeix-Raviart BEM that uniform mesh refinement is optimal for these kind of singularities.

The second example reports on a case where the right-hand side of our model problem is chosen to be singular, such that, due to the mapping properties of $\hyp$,
the exact solution $\phi$ suffers from low regularity as well. In this case, it will turn out that uniform mesh-refinement is suboptimal while adaptive refinement recovers the optimal rate.

\begin{figure}[t]
  \centering
  \psfrag{1}{}
  \psfrag{2}{}
  \psfrag{3}{}
  \psfrag{4}{}
  \psfrag{5}{}
  \psfrag{6}{}
  \psfrag{7}{}
  \psfrag{8}{}
  \includegraphics[width=6cm]{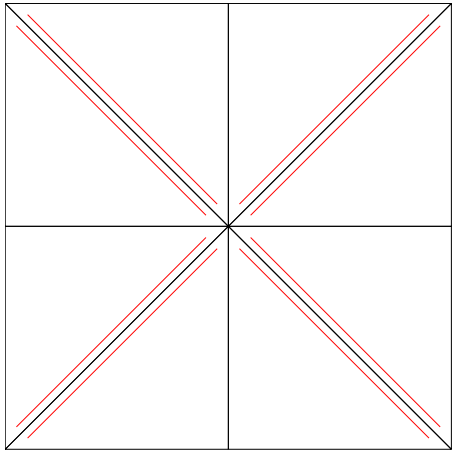}
  \caption{Initial mesh $\mesh_0$ used in the numerical experiments.}
  \label{fig:mesh0}
\end{figure}
\subsection{Experiment with smooth solution}\label{section:experiment1}
We consider the screen $\Gamma := \left[ 0,1 \right]^2$. The initial mesh $\mesh_0$ consists of $8$ congruent triangles, cf. Fig.~\ref{fig:mesh0},
such that $\Gamma$ is halved along the diagonals and the midpoints of its sides. The reference edges are chosen on the two diagonals.
The right-hand side is given by
\begin{align*}
  f(x,y)=1,
\end{align*}
and it is well known that the exact solution $\phi$ has square root edge singularities
\cite{Stephan_87_BIE} so that $\phi\in \wilde H^{1-\varepsilon}(\Gamma)$ for all $\varepsilon>0$.
We use the following five different sequences of meshes.

\textbf{Uniform sequence, Figure~\ref{fig:unif_smooth}.}
The sequence $\mesh_\ell$, $\ell\in\mathbb{N}_0$, is generated by uniform refinement, i.e.,
the initial mesh $\mesh_0$ is chosen as in Figure~\ref{fig:mesh0}, and the mesh
$\mesh_\ell$, $\ell\geq 1$ is generated from $\mesh_{\ell-1}$ by a bisec(3)-refinement
(as described in Figure~\ref{fig:nvb}) of every triangle $\el\in\mesh_\ell$.
Due to the results in~\cite{hs09}, we expect a convergence rate of $\OO(h_\ell^{1/2-\varepsilon}) = \OO(N_\ell^{-1/4+\varepsilon})$ for all $\varepsilon>0$, cf.~\eqref{eq:uniform}.
This is exactly what we observe in the convergence history in Fig.~\ref{fig:unif_smooth}.

\textbf{Adaptive sequence, Figure~\ref{fig:adap_smooth}.}
The sequence of meshes $\mesh_\ell$, $\ell\in\mathbb{N}_0$, is generated by Algorithm~\ref{algorithm} with $\theta=0.5$,
where $\mesh_0$ is chosen as in Figure~\ref{fig:mesh0}.
As we conjectured in Section~\ref{section:uniform}, the rate $\OO(N_\ell^{-1/4})$ cannot be improved in general, and this is what we see in the convergence history in Fig.~\ref{fig:adap_smooth}.
In Fig.~\ref{fig:adaptive_smooth_meshes} we plot the intermediate mesh $\mesh_{11}$ and the final mesh $\mesh_{22}$ that are constructed by the adaptive algorithm. What we observe
qualitatively is that the meshes are refined towards the boundary $\partial\Gamma$, which meets the expectation as $\phi$ exhibits singularities there.
Nevertheless, the computed meshes are not optimal for a conforming method.
This is visualized in Fig.~\ref{fig:adap_smooth}, where we also plot the conforming energy error $\enorm{\phi-\Phi_\ell^\CO}^2$. Clearly, we use the number of the degrees of freedom of the conforming method
for the x-axis.

\textbf{Graded sequence, Figs.~\ref{fig:graded_smooth_beta2} and~\ref{fig:graded_smooth_beta3}.}
We use a sequence of meshes $\mesh_\ell$, $\ell\in\mathbb{N}_0$ that is graded towards $\partial\Gamma$, i.e., for all elements $\el\in\mesh_\ell$ there holds
\begin{align*}
  h_\ell(\el) \simeq \dist(\el,\Gamma)^\beta.
\end{align*}
We select the parameters $\beta\in\left\{ 2,3 \right\}$.
The numerical results show that both gradings maintain the optimal rate for the Crouzeix-Raviart BEM, see Figs.~\ref{fig:graded_smooth_beta2} and ~\ref{fig:graded_smooth_beta3}.

\begin{figure}[t]
\centering
\psfrag{muTilde}[cr][cr]{\tiny $\widetilde\mu_\ell^2$}
\psfrag{eta}[cr][cr]{\tiny$\eta_\ell^2$}
\psfrag{coarse jumps inside}[cr][cr]{\tiny$\rho_\ell^2 + \widehat\rho_\ell^2$}
\psfrag{conforming error}[cr][cr]{\tiny$\enorm{\Phi_\ell - \Phi^0_\ell}^2$}
\psfrag{N^{-1/2}}[cr][cr]{\tiny$N_\ell^{-1/2}$}
\psfrag{N}[cc][cc]{Degrees of freedom}
\includegraphics{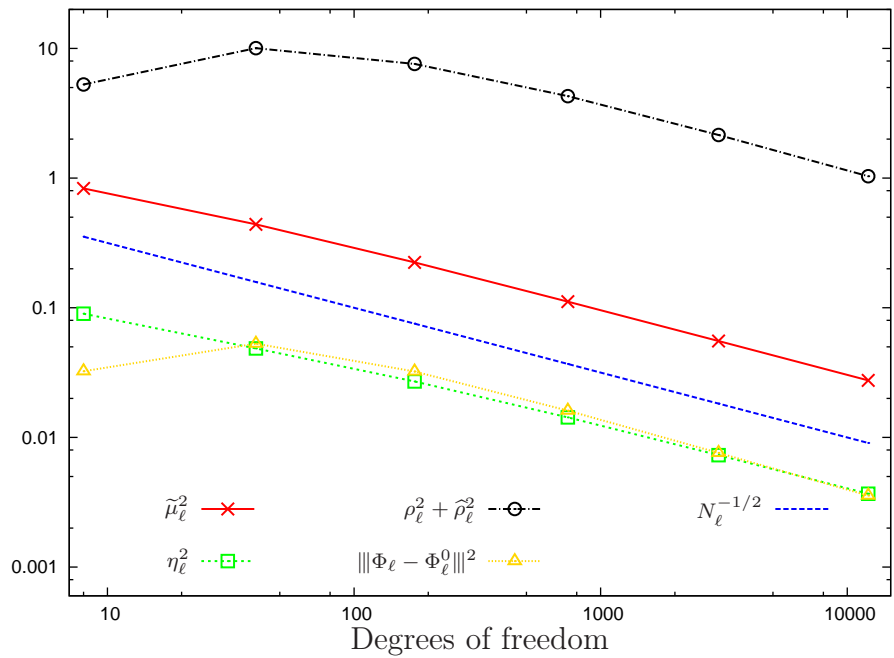}
\caption{Convergence history for uniform mesh refinement and smooth solution. We see that the squared quantities exhibit the optimal rate $\OO(N_\ell^{-1/2})$.}
\label{fig:unif_smooth}
\end{figure}
\begin{figure}[t]
\centering
\psfrag{muTilde}[cr][cr]{\tiny$\widetilde\mu_\ell^2$}
\psfrag{eta}[cr][cr]{\tiny$\eta_\ell^2$}
\psfrag{coarse jumps inside}[cr][cr]{\tiny$\rho_\ell^2 + \widehat\rho_\ell^2$}
\psfrag{conforming error}[cr][cr]{\tiny$\enorm{\Phi_\ell - \Phi^0_\ell}^2$}
\psfrag{nrg err}[cr][cr]{\tiny$\enorm{\phi-\Phi_\ell^\CO}^2$}
\psfrag{N^{-1/2}}[cr][cr]{\tiny$N_\ell^{-1/2}$}
\psfrag{N}[cc][cc]{Degrees of freedom}
\includegraphics{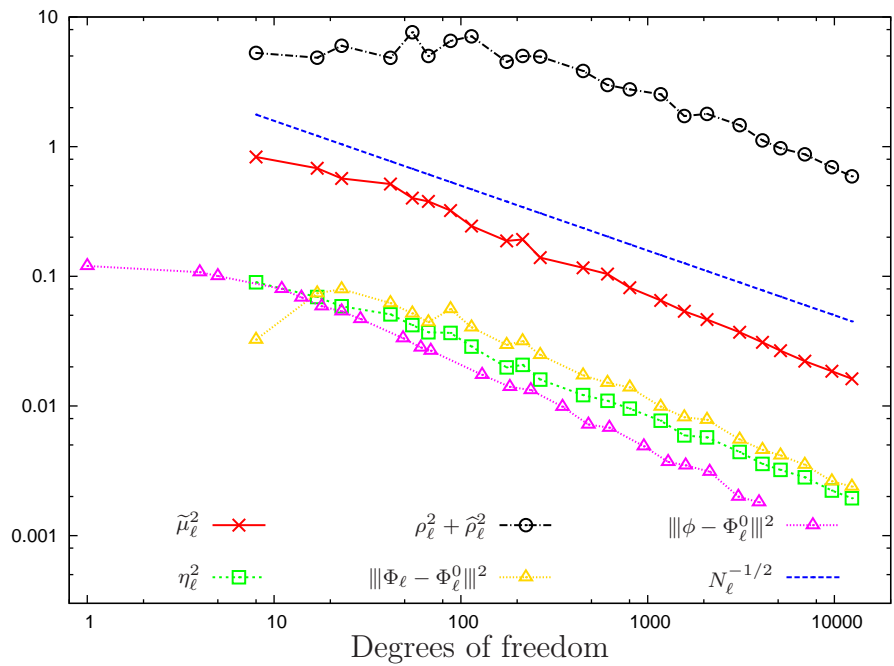}
\caption{Convergence history for adaptive algorithm and smooth solution. We see that the squared quantities exhibit the optimal rate $\OO(N_\ell^{-1/2})$.}
\label{fig:adap_smooth}
\end{figure}
\begin{figure}[t]
\centering
\psfrag{muTilde}[cr][cr]{\tiny $\widetilde\mu_\ell^2$}
\psfrag{eta}[cr][cr]{\tiny$\eta_\ell^2$}
\psfrag{jumps}[cr][cr]{\tiny$\rho_\ell^2 + \widehat\rho_\ell^2$}
\psfrag{conforming error}[cr][cr]{\tiny$\enorm{\Phi_\ell - \Phi^0_\ell}^2$}
\psfrag{N^{-1/2}}[cr][cr]{\tiny$N_\ell^{-1/2}$}
\psfrag{N}[cc][cc]{Degrees of freedom}
\includegraphics{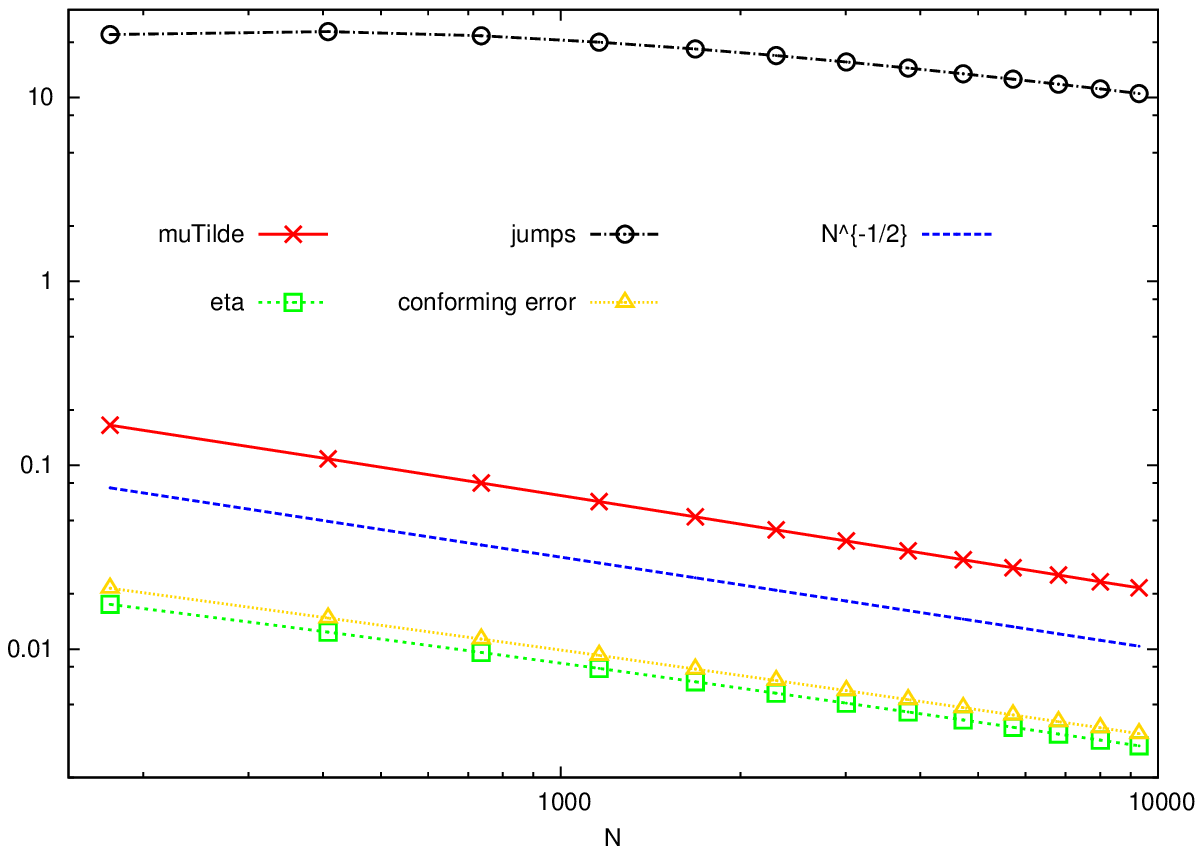}
\caption{Convergence history for graded meshes with $\beta=2$ and smooth solution. We see that the squared quantities exhibit the optimal rate $\OO(N_{\ell}^{-1/2})$.}
\label{fig:graded_smooth_beta2}
\end{figure}
\begin{figure}[t]
\centering
\psfrag{muTilde}[cr][cr]{\tiny$\widetilde\mu_\ell^2$}
\psfrag{eta}[cr][cr]{\tiny$\eta_\ell^2$}
\psfrag{jumps}[cr][cr]{\tiny$\rho_\ell^2 + \widehat\rho_\ell^2$}
\psfrag{conforming error}[cr][cr]{\tiny$\enorm{\Phi_\ell - \Phi^0_\ell}^2$}
\psfrag{N^{-1/2}}[cr][cr]{\tiny$N_\ell^{-1/2}$}
\psfrag{N}[cc][cc]{Degrees of freedom}
\includegraphics{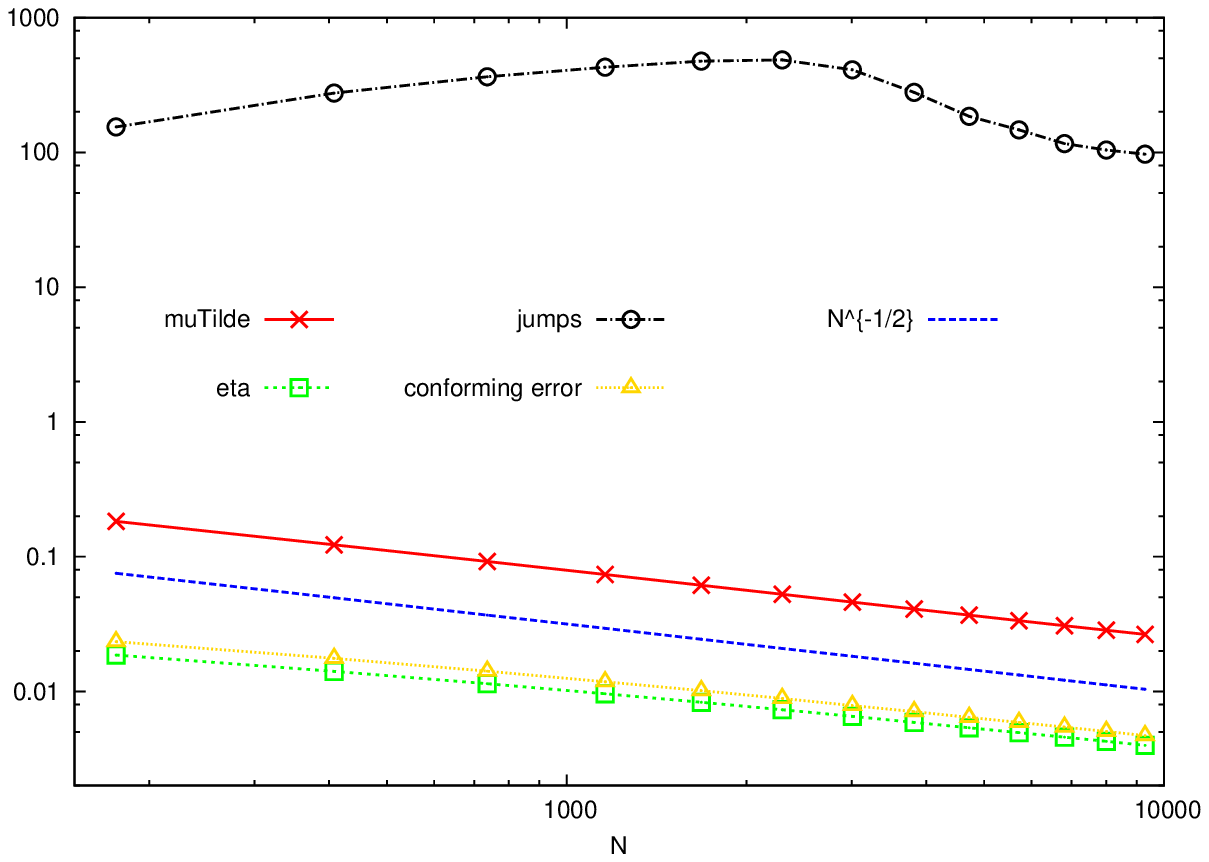}
\caption{Convergence history for graded meshes with $\beta=3$ and smooth solution. We see that the squared quantities exhibit the optimal rate $\OO(N_{\ell}^{-1/2})$.}
\label{fig:graded_smooth_beta3}
\end{figure}
\begin{figure}[t]
  \centering
  \includegraphics[width=6cm]{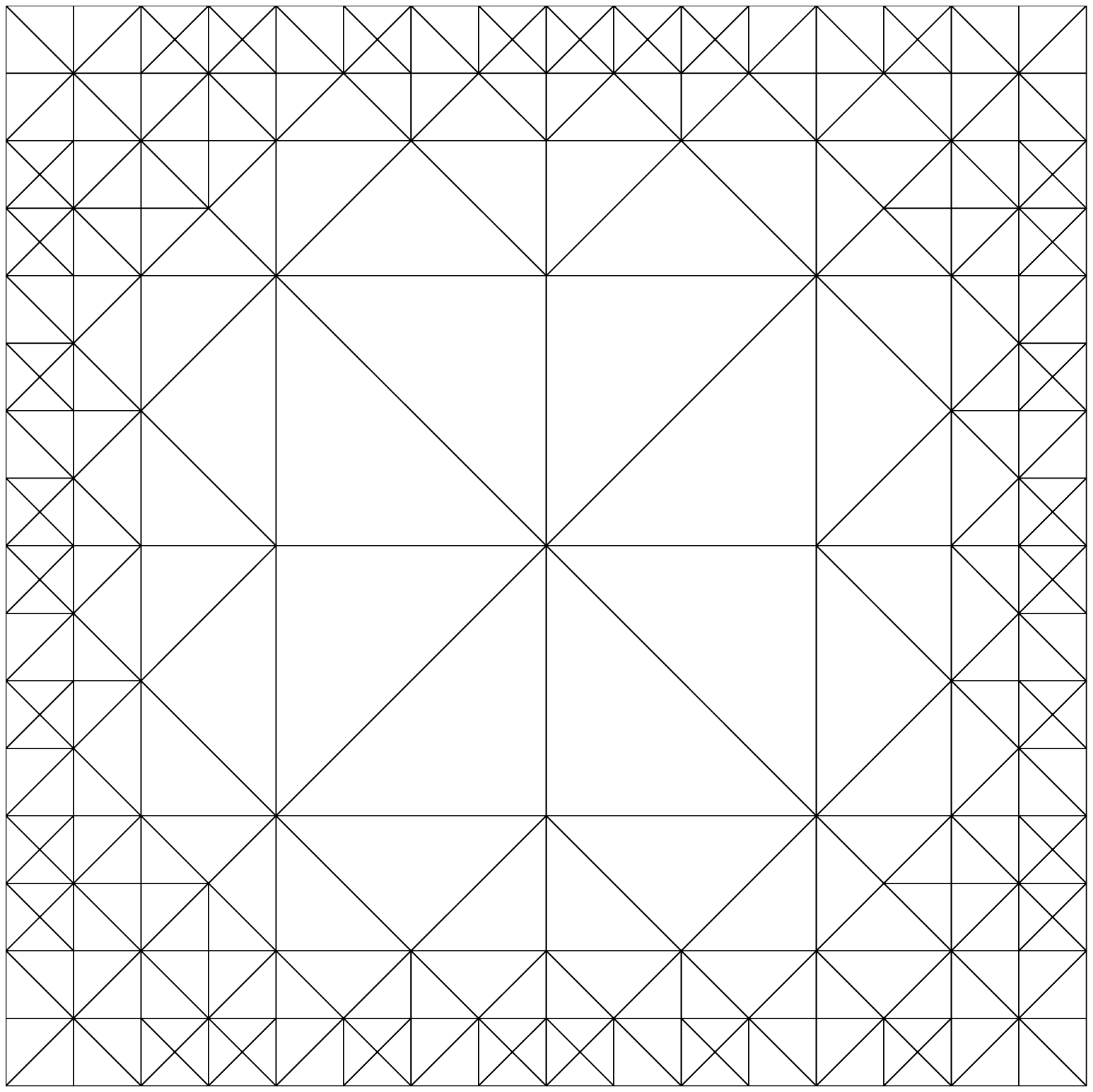}
  \includegraphics[width=6cm]{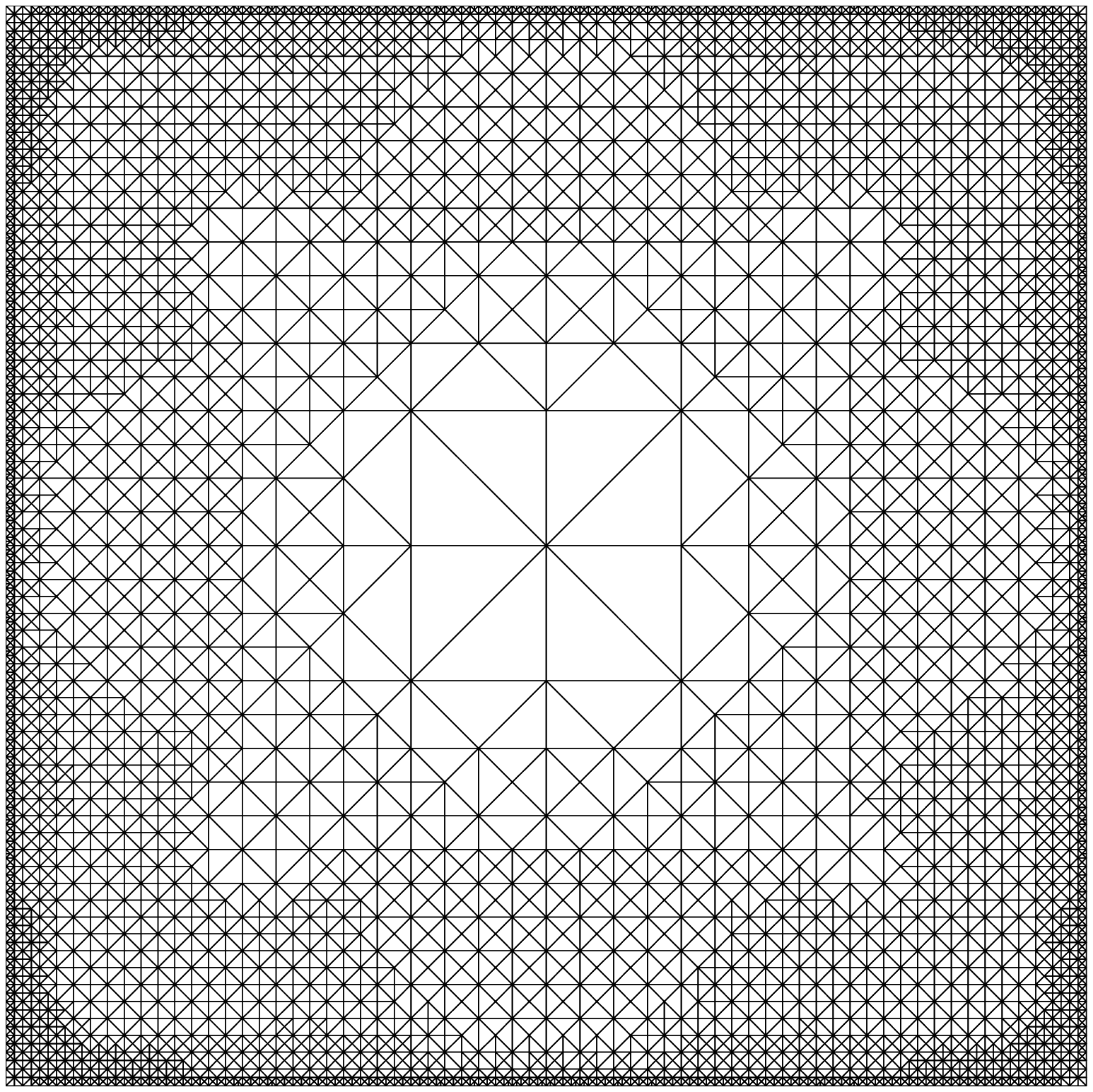}
  \caption{Meshes $\mesh_{11}$ and $\mesh_{22}$ of the adaptive algorithm for smooth solution.}
  \label{fig:adaptive_smooth_meshes}
\end{figure}
\subsection{Experiment with singular solution}
The right-hand side is given by
\begin{align*}
  f(x,y) := x^{-6/10},
\end{align*}
and because of $f\notin L_2(\Gamma)$ we conclude from the mapping properties of $\hyp$ that the exact solution fulfills $\phi\notin H^1(\Gamma)$.
The missing regularity will lead to a suboptimal convergence rate for uniform refinement, which will be recovered by the adaptive algorithm.
Let us briefly discuss what to expect in the uniform case: for the function $g(x)=x^\alpha$ there holds $g\in H^{\alpha+1/2-\varepsilon}(0,1)\setminus H^{\alpha+1/2}(0,1)$ for all $\varepsilon>0$.
We conclude that, $f\in H^{-0.1-\varepsilon}(\Gamma)\setminus H^{-0.1}(\Gamma)$, and due to the mapping properties of $\hyp$ we conclude that $\phi\notin \wilde H^{9/10}(\Gamma)$. Hence, we
expect a convergence rate which is worse than $\OO(h_\ell^{4/10}) = \OO(N_\ell^{-1/5})$ for uniform refinement.
We already stated the choice of the initial mesh $\mesh_0$. Uniform and adaptive meshes are computed exactly as described in Section~\ref{section:experiment1}. The convergence history
for the uniform sequence of meshes is depicted in Fig.~\ref{fig:unif_singular}. We see that the uniform scheme is suboptimal, and the the convergence rate is indeed worse than
$\OO(N_\ell^{-1/5})$ (note that we plot squared quantities). However, the adaptive sequence of meshes, depicted in Fig.~\ref{fig:adap_singular}, recovers the optimal convergence rate.
In Fig.~\ref{fig:adaptive_singular_meshes}, we plot the two adaptive meshes $\mesh_{11}$ and $\mesh_{23}$ which are generated by the adaptive algorithm.

\begin{figure}[ht]
\centering
\psfrag{muTilde}[cr][cr]{\tiny $\widetilde\mu_\ell^2$}
\psfrag{eta}[cr][cr]{\tiny$\eta_\ell^2$}
\psfrag{coarse jumps inside}[cr][cr]{\tiny$\rho_\ell^2 + \widehat\rho_\ell^2$}
\psfrag{conforming error}[cr][cr]{\tiny$\enorm{\Phi_\ell - \Phi^0_\ell}^2$}
\psfrag{N^{-1/2}}[cr][cr]{\tiny$N_\ell^{-1/2}$}
\psfrag{N^{-2/5}}[cr][cr]{\tiny$N_\ell^{-2/5}$}
\psfrag{N}[cc][cc]{Degrees of freedom}
\includegraphics{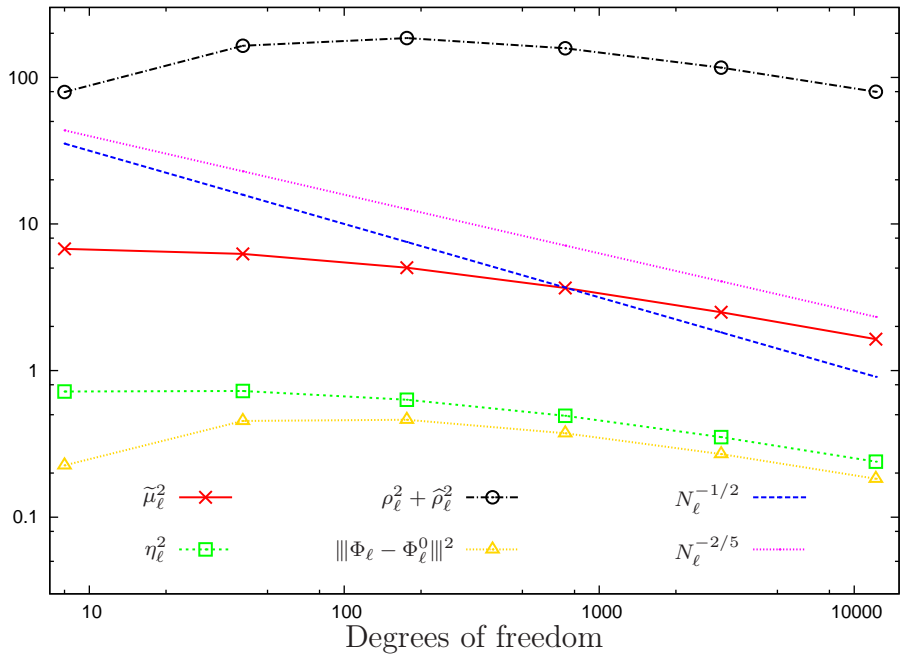}
\caption{Convergence history for uniform mesh refinement and singular right-hand side.
The squared quantities do not exhibit the optimal rate, which would be $\OO(N_{\ell}^{-1/2})$.}
\label{fig:unif_singular}
\end{figure}
\begin{figure}[t]
\centering
\psfrag{muTilde}[cr][cr]{\tiny$\widetilde\mu_\ell^2$}
\psfrag{eta}[cr][cr]{\tiny$\eta_\ell^2$}
\psfrag{coarse jumps inside}[cr][cr]{\tiny$\rho_\ell^2 + \widehat\rho_\ell^2$}
\psfrag{conforming error}[cr][cr]{\tiny$\enorm{\Phi_\ell - \Phi^0_\ell}^2$}
\psfrag{N^{-1/2}}[cr][cr]{\tiny$N_\ell^{-1/2}$}
\psfrag{N}[cc][cc]{Degrees of freedom}
\includegraphics{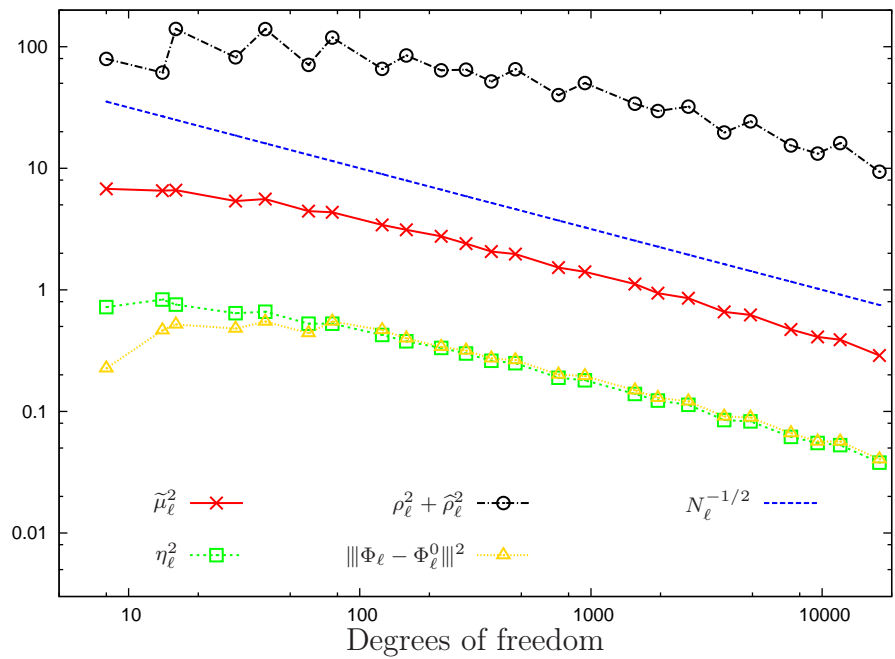}
\caption{Convergence history for adaptive algorithm and singular right-hand side.
The squared quantities exhibit the optimal rate $\OO(N_{\ell}^{-1/2})$.}
\label{fig:adap_singular}
\end{figure}

\begin{figure}[ht]
  \centering
  \includegraphics[width=6cm]{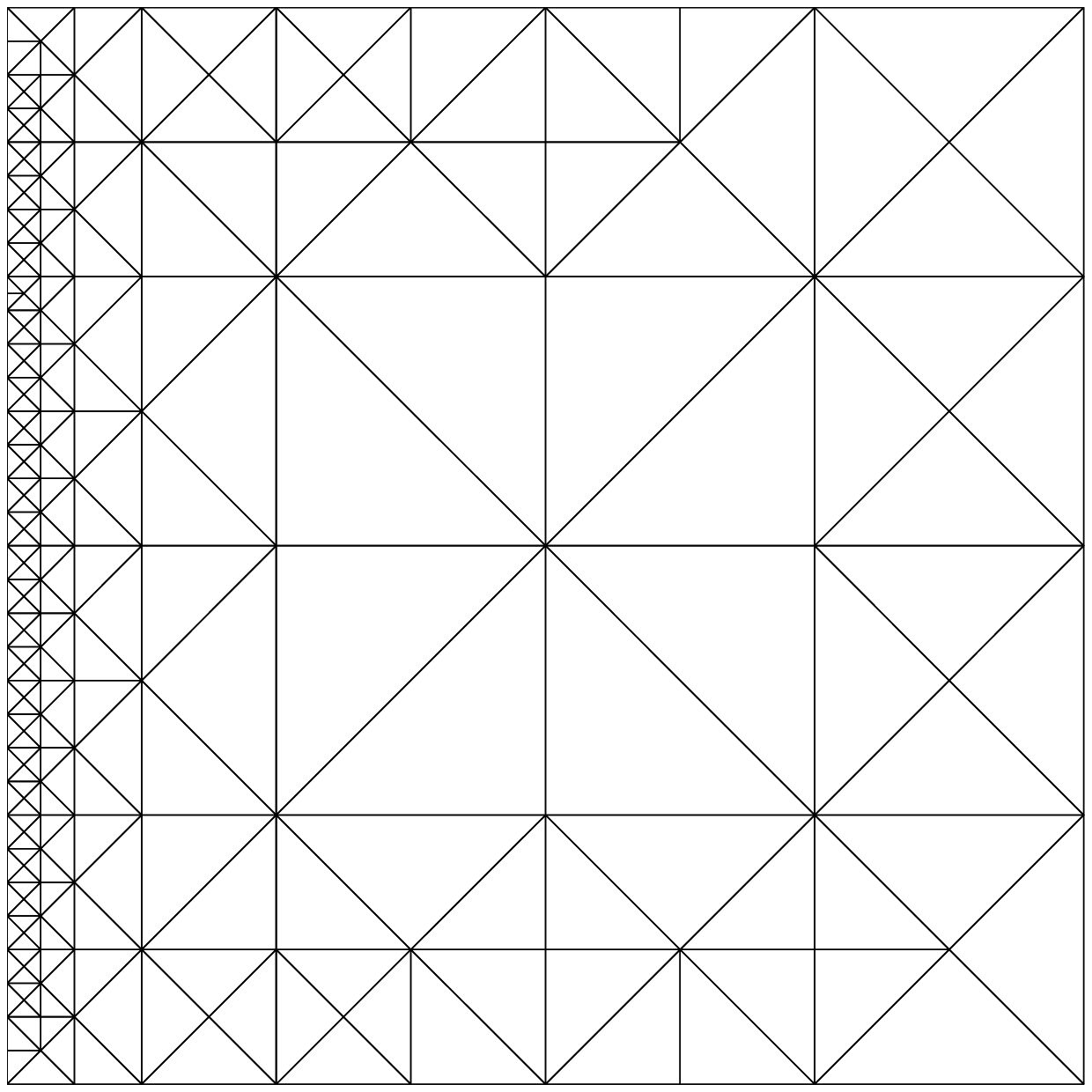}
  \includegraphics[width=6cm]{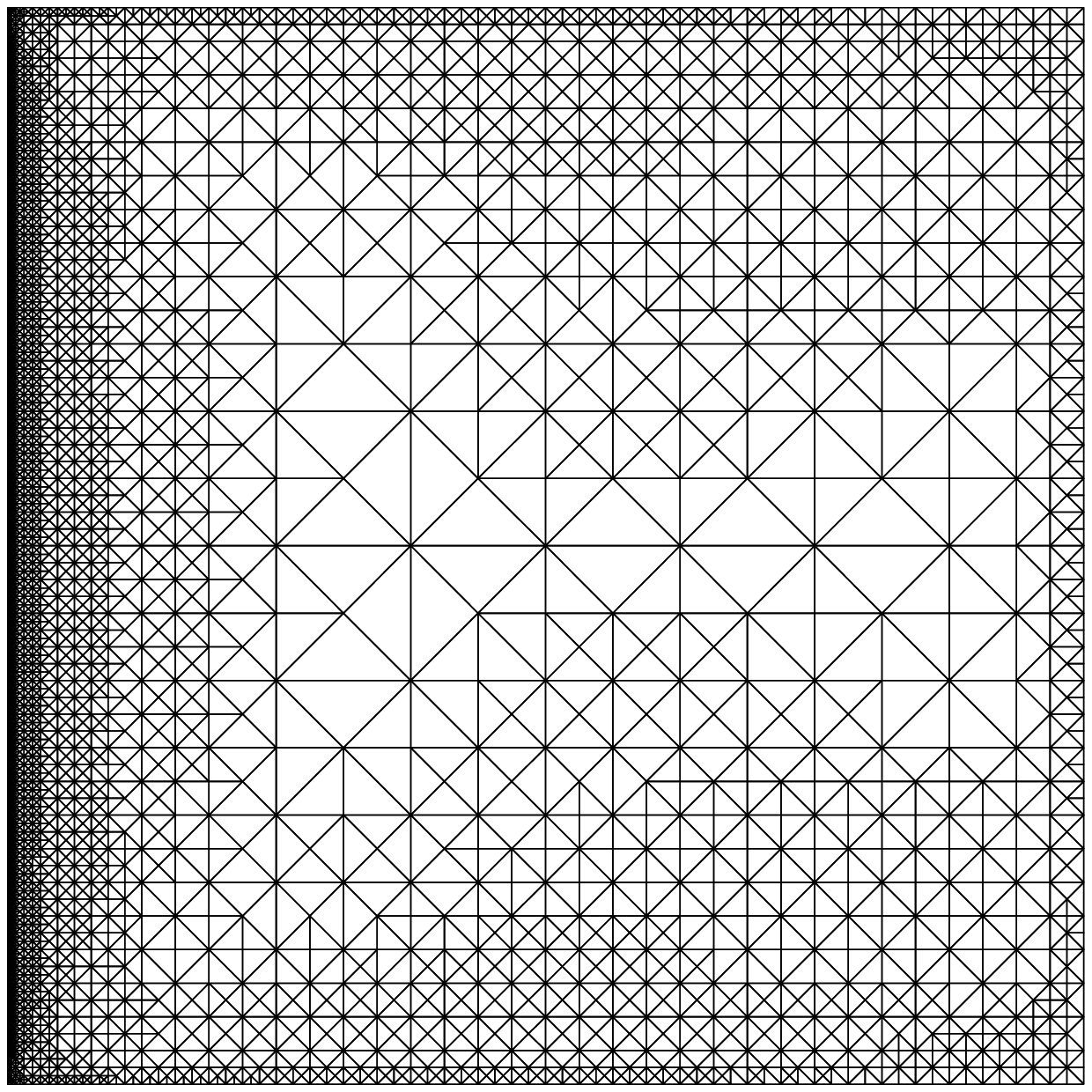}
  \caption{Meshes $\mesh_{11}$ and $\mesh_{23}$ of the adaptive algorithm for singular solution.}
  \label{fig:adaptive_singular_meshes}
\end{figure}

\bibliographystyle{alpha}
\bibliography{literature}
\end{document}